\documentclass[12pt]{article}
\usepackage{amsmath,enumerate,amsfonts,amssymb,color,graphicx,amsthm}

\usepackage[colorlinks=true, allcolors=black]{hyperref}
\usepackage{todonotes}
\usepackage[normalem]{ulem}
\numberwithin{equation}{section}
\usepackage{cite}
\usepackage{mathtools}
\usepackage[nottoc,notlot,notlof]{tocbibind}
\usepackage{appendix}

\usepackage{lipsum}

\def\XXint#1#2#3{{\setbox0=\hbox{$#1{#2#3}{\int}$ }
		\vcenter{\hbox{$#2#3$ }}\kern-.6\wd0}}

\newlength{\leftstackrelawd}
\newlength{\leftstackrelbwd}
\def\leftstackrel#1#2{\settowidth{\leftstackrelawd}%
	{${{}^{#1}}$}\settowidth{\leftstackrelbwd}{$#2$}%
	\addtolength{\leftstackrelawd}{-\leftstackrelbwd}%
	\leavevmode\ifthenelse{\lengthtest{\leftstackrelawd>0pt}}%
	{\kern-.5\leftstackrelawd}{}\mathrel{\mathop{#2}\limits^{#1}}}

\theoremstyle{plain}
\newtheorem{thm}{Theorem}[section]
\newtheorem{lem}[thm]{Lemma}
\newtheorem{cor}[thm]{Corollary}
\newtheorem{prop}[thm]{Proposition}
\newtheorem*{thm*}{Theorem}

\theoremstyle{definition}
\newtheorem{defn}[thm]{Definition}

\newtheorem{rmk}[thm]{Remark}
\newtheorem{?}[thm]{Problem}

\newtheorem{innercustomthm}{Theorem}
\newenvironment{customthm}[1]
{\renewcommand\theinnercustomthm{#1}\innercustomthm}
{\endinnercustomthm}

\newcommand{\KN}{\mathbin{\bigcirc\mspace{-15mu}\wedge\mspace{3mu}}}
\newcommand{\ep}{\varepsilon}
\renewcommand{\phi}{\varphi}
\renewcommand{\epsilon}{\varepsilon}
\makeatletter
\def\@cite#1#2{[\textbf{#1\if@tempswa , #2\fi}]}
\def\@biblabel#1{[\textbf{#1}]}
\makeatother

\makeatletter
\newcommand*{\defeq}{\mathrel{\rlap{%
			\raisebox{0.3ex}{$\m@th\cdot$}}%
		\raisebox{-0.3ex}{$\m@th\cdot$}}%
	=}
\makeatother

\makeatletter
\newcommand*{\eqdef}{=\mathrel{\rlap{%
			\raisebox{0.3ex}{$\m@th\cdot$}}%
		\raisebox{-0.3ex}{$\m@th\cdot$}}%
	}
\makeatother

\newcounter{marnote}

\makeatletter
\def\underbracex#1#2{\mathop{\vtop{\m@th\ialign{##\crcr
				$\hfil\displaystyle{#2}\hfil$\crcr
				\noalign{\kern3\p@\nointerlineskip}%
				#1\crcr\noalign{\kern3\p@}}}}\limits}

\def\upbracefilla{$\m@th \setbox\z@\hbox{$\braceld$}%
	\bracelu\leaders\vrule \@height\ht\z@ \@depth\z@\hfill 
	\kern\p@\vrule \@width\p@\kern\p@\vrule \@width\p@\kern\p@\vrule \@width\p@
	$}

\def\upbracefillb{$\m@th \setbox\z@\hbox{$\braceld$}%
	\vrule \@width\p@\kern\p@\vrule \@width\p@\kern\p@\vrule \@width\p@\kern\p@
	\leaders\vrule \@height\ht\z@ \@depth\z@\hfill\bracerd
	\braceld\leaders\vrule \@height\ht\z@ \@depth\z@\hfill
	\kern\p@\vrule \@width\p@\kern\p@\vrule \@width\p@\kern\p@\vrule \@width\p@
	$}

\def\upbracefillc{$\m@th \setbox\z@\hbox{$\braceld$}%
	\vrule \@width\p@\kern\p@\vrule \@width\p@\kern\p@\vrule \@width\p@\kern\p@
	\leaders\vrule \@height\ht\z@ \@depth\z@\hfill
	\kern\p@\vrule \@width\p@\kern\p@\vrule \@width\p@\kern\p@\vrule \@width\p@
	$}

\def\upbracefilld{$\m@th \setbox\z@\hbox{$\braceld$}%
	\vrule \@width\p@\kern\p@\vrule \@width\p@\kern\p@\vrule \@width\p@\kern\p@
	\leaders\vrule \@height\ht\z@ \@depth\z@\hfill\braceru$}

\def\upbracefillbd{$\m@th \setbox\z@\hbox{$\braceld$}%
	\vrule \@width\p@\kern\p@\vrule \@width\p@\kern\p@\vrule \@width\p@\kern\p@
	\bracerd\braceld
	\leaders\vrule \@height\ht\z@ \@depth\z@\hfill\braceru$}

\makeatother
\date{}

\begin{document}

\title{The $\sigma_k$-Loewner-Nirenberg problem on Riemannian manifolds for $k=\frac{n}{2}$ and beyond}
\author{Jonah A. J. Duncan\footnote{Department of Mathematics, University College London, 25 Gordon Street, London, WC1H 0AY, UK. Email: jonah.duncan@ucl.ac.uk. Supported by the Additional Funding Programme for Mathematical Sciences, delivered by EPSRC (EP/V521917/1) and the Heilbronn Institute for Mathematical Research.} ~and Luc Nguyen\footnote{Mathematical Institute and St Edmund Hall, University of Oxford, Andrew Wiles Building, Radcliffe Observatory Quarter, Woodstock Road, OX2 6GG, UK. Email: luc.nguyen@maths.ox.ac.uk}}
\maketitle

\begin{abstract}
	Let $(M^n,g_0)$ be a smooth compact Riemannian manifold of dimension $n\geq 3$ with smooth non-empty boundary $\partial M$. Let $\Gamma\subset\mathbb{R}^n$ be a symmetric convex cone and $f$ a symmetric defining function for $\Gamma$ satisfying standard assumptions. Denoting by $A_{g_u}$ the Schouten tensor of a conformal metric $g_u = u^{-2}g_0$, we show that the associated fully nonlinear Loewner-Nirenberg problem 
	\begin{align*}
	\begin{cases}
	f(\lambda(-g_u^{-1}A_{g_u})) = \frac{1}{2}, \quad \lambda(-g_u^{-1}A_{g_u})\in\Gamma & \text{on }M\backslash \partial M \\
	u = 0 & \text{on }\partial M
	\end{cases}
	\end{align*}
	admits a solution if $\mu_\Gamma^+ > 1-\delta$, where $\mu_\Gamma^+$ is defined by $(-\mu_\Gamma^+,1,\dots,1)\in\partial\Gamma$ and $\delta>0$ is a constant depending on certain geometric data. In particular, we solve the $\sigma_k$-Loewner-Nirenberg problem for all $k\leq  \frac{n}{2}$, which extends recent work of the authors to include the important threshold case $k=\frac{n}{2}$. In the process, we establish that the fully nonlinear Loewner-Nirenberg problem and corresponding Dirichlet boundary value problem with positive boundary data admit solutions if there exists a conformal metric $g\in[g_0]$ such that $\lambda(-g^{-1}A_g)\in\Gamma$ on $M$; these latter results require no assumption on $\mu_\Gamma^+$ and are new when $(1,0,\dots,0)\in\partial\Gamma$.  \vspace*{3mm}
\end{abstract}

\setcounter{tocdepth}{2}
\tableofcontents

\section{Introduction}

In their classical work \cite{LN74}, Loewner \& Nirenberg established (among other results) that a smooth bounded domain in $\mathbb{R}^n$ ($n\geq 3$) admits a smooth complete conformally flat metric with constant negative scalar curvature. Aviles \& McOwen \cite{AM88} later extended this result to compact Riemannian manifolds with boundary. Since \cite{AM88, LN74} there have been a number of related works addressing questions such as asymptotics of solutions, existence in non-smooth domains and applications to mathematical physics; for a partial list of references see Allen et.~al.~\cite{AILA18}, Andersson et.~al.~\cite{ACF92}, Aviles \cite{Av82}, Finn \cite{Finn98}, Gover \& Waldron \cite{GW17}, Graham \cite{Gr17}, Gursky \& Graham \cite{GG21}, Han et.~al.~\cite{HJS20}, Han \& Shen \cite{HS20}, Hogg \& Nguyen \cite{HN23}, Jiang \cite{Jia21}, Li \cite{Li22}, Mazzeo \cite{Maz91} and V\'eron \cite{Ver81}. \medskip 

A natural question is whether these results can be extended to other notions of negative curvature.  Since the pioneering works of Caffarelli, Nirenberg \& Spruck \cite{CNS3}, Viaclovsky \cite{Via00a} and Chang, Gursky \& Yang \cite{CGY02a}, a central problem has been to prescribe certain symmetric functions of the eigenvalues of the Schouten tensor within a conformal class. Recall that the Schouten tensor of a Riemannian metric $g$ is defined by
\begin{align*}
A_g = \frac{1}{n-2}\bigg(\operatorname{Ric}_g - \frac{R_g}{2(n-1)}g\bigg),
\end{align*} 
where $\operatorname{Ric}_g$ and $R_g$ denote the Ricci curvature and scalar curvature of $g$, respectively. More precisely, suppose that $f$ and $\Gamma$ satisfy the following properties: 
\begin{align}
& \Gamma\subset\mathbb{R}^n\text{ is an open, convex, connected symmetric cone with vertex at 0}, \label{21'} \\
& \Gamma_n^+ = \{\lambda\in\mathbb{R}^n: \lambda_i > 0 ~\forall ~1\leq i \leq n\} \subseteq \Gamma \subseteq \Gamma_1^+ =  \{\lambda\in\mathbb{R}^n : \lambda_1+\dots+\lambda_n > 0\}, \label{22'} \\
& f\in C^\infty(\Gamma)\cap C^0(\overline{\Gamma}) \text{ is concave, homogeneous of degree one and symmetric in the }\lambda_i, \label{23'}  \\
& f>0 \text{ in }\Gamma, \quad f = 0 \text{ on }\partial\Gamma, \quad f_{\lambda_i} >0 \text{ in } \Gamma \text{ for }1 \leq i \leq n. \label{24'}
\end{align}
When convenient we may also assume without loss of generality that $f$ is normalised so that 
\begin{align}\label{25'}
f(e) = 1, \quad e = (1,\dots,1).
\end{align}
One is then led to the following generalisation of the Loewner-Nirenberg problem:\medskip 

\noindent\textbf{The fully nonlinear Loewner-Nirenberg problem:} \textit{Given a smooth compact Riemannian manifold $(M,g_0)$ of dimension $n\geq 3$ with smooth non-empty boundary, and given $(f,\Gamma)$ satisfying \eqref{21'}--\eqref{25'}, does there exist a conformal metric $g_u = u^{-2}g_0$ which is complete in the interior of $M$ and satisfies}
\begin{align}\label{105}
f(\lambda(-g_u^{-1}A_{g_u})) = \frac{1}{2}, \quad \lambda(-g_u^{-1}A_{g_u})\in\Gamma \quad \text{on }M 
\end{align}
\textit{and}
\begin{align}\label{300}
\lim_{\mathrm{d}_{g_0}(x,\partial M)\rightarrow 0}\, \frac{u(x)}{\mathrm{d}_{g_0}(x,\partial M)}\in(0,\infty).
\end{align}

Typical examples of $(f,\Gamma)$ satisfying \eqref{21'}--\eqref{25'} are given by $(c\sigma_k^{1/k}, \Gamma_k^+)$ for $1\leq k \leq n$, where $c = c_{n,k} =\binom{n}{k}^{-1/k}$, $\sigma_k:\mathbb{R}^n\rightarrow\mathbb{R}$ is the $k$'th elementary symmetric polynomial and $\Gamma_k^+$ is the G{\aa}rding cone:
\begin{align*}
\sigma_k(\lambda_1,\dots,\lambda_n) = \sum_{1\leq i_1 < \dots< i_k \leq n} \lambda_{i_1}\dots\lambda_{i_k} \quad \text{and} \quad \Gamma_k^+ = \{\lambda \in\mathbb{R}^n : \sigma_j(\lambda)>0 \text{ for }1 \leq j \leq k\}.
\end{align*}
In these cases the fully nonlinear Loewner-Nirenberg problem is also referred to as the $\sigma_k$-Loewner-Nirenberg problem. Note that, since the trace of the Schouten tensor is a positive multiple of the scalar curvature, the $\sigma_1$-Loewner-Nirenberg problem is the original problem considered by Loewner \& Nirenberg on Euclidean domains.\medskip

The fully nonlinear Loewner-Nirenberg problem is a type of uniformisation problem within a fixed conformal class, well-motivated from both a geometric and a PDE perspective. Indeed, the Schouten tensor arises as the trace part of the Riemann curvature tensor in the so-called Ricci decomposition:
\begin{align}\label{101}
\operatorname{Riem}_g = W_g + A_g \KN g.
\end{align}
Here, $W_g$ is the $(0,4)$-Weyl tensor of $g$ (which is trace-free) and $\KN$ is the Kulkarni-Nomizu product. In light of the conformal invariance of $g^{-1}W_g$, \eqref{101} demonstrates that the conformal transformation properties of $\operatorname{Riem}_g$ are completely determined by those of $A_g$. Moreover, as shown by Li \& Li in \cite{LL03}, the Schouten tensor plays a central role in the characterisation of conformally invariant operators on $\mathbb{R}^n$ (see also \cite{LLL21, LN09b}). From a PDE point of view, if $g_w = w^{-2}g_0$ then $A_{g_w}$ and $A_{g_0}$ are related by 
\begin{align}\label{104}
A_{g_w} = w^{-1}\nabla_{g_0}^2w - \frac{1}{2}w^{-2}|dw|_{g_0}^2\, g_0 + A_{g_0}
\end{align}
and thus \eqref{105} is a fully nonlinear, non-uniformly elliptic equation, similar in structure to the Hessian equations that have received significant attention since the seminal work of Caffarelli, Nirenberg \& Spruck \cite{CNS3}. Due to the lower order terms in \eqref{104}, \eqref{105} is also recognised as an augmented Hessian equation. In the study of these equations, the so-called MTW condition, introduced by Ma, Trudinger \& Wang in \cite{MTW}, has played an important role (see e.g.~\cite{JT17, JT18, JT19} in the Euclidean setting and the references therein). We point out that \eqref{105} does \textit{not} satisfy the MTW condition when $f\not=c\sigma_1$ due to the minus sign in front of the gradient term in \eqref{104}. For some recent work on augmented Hessian equations on Riemannian manifolds, see e.g.~Duncan \cite{D23}, Guan \cite{Guan14, Guan23} and Guan \& Jiao \cite{GJ15, GJ16}.\medskip

When $M\subset\mathbb{R}^n$ is a domain with smooth boundary and $g_0$ is the Euclidean metric, the fully nonlinear Loewner-Nirenberg problem was solved by Gonz\'alez, Li \& Nguyen \cite{GLN18}. However, current existence results on general Riemannian manifolds are sensitive to the value $\mu_\Gamma^+$, defined by Li \& Nguyen \cite{LN14b} to be the quantity satisfying
\begin{align*}
(-\mu_\Gamma^+,1,\dots,1)\in\partial\Gamma.
\end{align*}
Indeed, on an arbitrary compact Riemannian manifold with non-empty boundary, the existence of a solution to the fully nonlinear Loewner-Nirenberg problem is only known when $\mu_\Gamma^+>1$, due to previous work of the authors in \cite{DN23}. See also the combination of results of Yuan \cite{Yuan22, Yuan24} for related work in the case $\mu_\Gamma^+ \geq 1$ under the additional assumption $(1,0,\dots,0)\in\Gamma$. Note that $\mu_{\Gamma_k^+}^+ = \frac{n-k}{k}$ and hence $\mu_{\Gamma_k^+}^+ > 1$ if and only if $k< \frac{n}{2}$. More generally, by \eqref{21'} and \eqref{22'}, $\mu_\Gamma^+$ is well-defined, uniquely determined by $\Gamma$ and satisfies $\mu_\Gamma^+\in[0,n-1]$. \medskip 

The main result of this paper (see Theorem \ref{A'} below) extends our existence result in \cite{DN23} to the case $\mu_\Gamma^+ = 1$ without any further assumption on the cone $\Gamma$. This is of significance on both geometric and analytical grounds. Geometrically, the case $\Gamma = \Gamma_{n/2}^+$ in even dimensions is particularly interesting due to the relation between the $\sigma_{n/2}$-curvature, the Pfaffian of the curvature form and the Chern-Gauss-Bonnet formula -- see e.g.~the pioneering work of Chang, Gursky \& Yang \cite{CGY02a, CGY02b} in the context of positive curvature, and the more recent work of Graham \& Gursky \cite{GG21} in the negative curvature setting. There is also a close relationship between the constraint $\lambda(\pm g^{-1}A_g)\in\Gamma$ when $\mu_\Gamma^+ \leq 1$ and the sign of the Ricci curvature of $g$ -- see e.g.~\cite{GVW03, GV06, LN14}. The case $\mu_\Gamma^+=1$ also presents analytical challenges due to the recently observed failure of certain estimates. In \cite{CLL23}, Chu, Li \& Li  observed the failure of local interior gradient estimates depending on one-sided $C^0$ bounds and the failure of the Liouville theorem on $\mathbb{R}^n$ precisely when $\mu_\Gamma^+ = 1$. More recently, the authors observed in \cite{DN25a} that local boundary $C^0$ estimates fail when $\mu_\Gamma^+ \leq 1$, and the statement of the Liouville theorem on $\mathbb{R}^n_+$ is fundamentally different in the cases $\mu_\Gamma^+ \leq 1$ and $\mu_\Gamma^+ >1$. In particular, the hyperbolic metric is the unique solution to the fully nonlinear Loewner-Nirenberg problem on $\mathbb{R}_+^n$ if and only if $\mu_\Gamma^+ >1$. \medskip 

In fact, rather than just addressing the case $\mu_\Gamma^+ =1$, we prove a more quantitative statement. Given a Riemannian manifold $(M,g)$ with non-empty boundary $\partial M$, we let $i_{(M,g)}$ denote the injectivity radius of $(M,g)$, $i_{(\partial M,g)}$ the injectivity radius of $\partial M$ equipped with the induced metric (which we also denote by $g$), $i^b_{(M,g)}$ the boundary injectivity radius of $(M,g)$ (see e.g.~\cite[Section 3]{AKKLT04} for the definition) and $H$ the mean curvature of $\partial M$.

\begin{defn}\label{500}
	For $R_0, S_0 \geq 0$, $i_0, d_0>0$ and $\sigma\in(0,1)$, let $\mathcal{M}^n_\sigma(R_0, i_0, S_0,d_0)$ denote the set of smooth $n$-dimensional Riemannian manifolds $(M, g)$ with non-empty boundary $\partial M$ such that
	\begin{align}\label{44}
	-R_0g \leq \operatorname{Ric}_{(M,g)} \leq R_0 g \quad \text{and} \quad -R_0 g \leq \operatorname{Ric}_{(\partial M, g)} \leq R_0 g, 
	\end{align}
	\begin{align}\label{45}
	i_{(M,g)} \geq i_0, \quad i_{(\partial M, g)} \geq i_0 \quad \text{and} \quad i^b_{(M,g)} \geq i_0, 
	\end{align}
	\begin{align}\label{46}
	\|H\|_{C^{0,\sigma}(\partial M, g)} \leq S_0
	\end{align}
	and
	\begin{align}\label{49}
	\operatorname{diam}(M,g) \leq d_0.
	\end{align}
\end{defn}

Our first main result is as follows:
\begin{thm}\label{A'}
	Let $n\geq 3$, $R_0, S_0 \geq 0$, $i_0, d_0>0$ and $\sigma\in(0,1)$. There exists a constant $\delta = \delta(n, R_0, i_0, S_0, d_0, \sigma)>0$ such that for any $(M,g_0)\in \mathcal{M}_\sigma^n(R_0, i_0, S_0, d_0)$ and any $(f,\Gamma)$ satisfying \eqref{21'}--\eqref{25'} and $\mu_\Gamma^+ > 1-\delta$, there exists a maximal locally Lipschitz viscosity solution $g_u = u^{-2}g_0$ to \eqref{105} satisfying 
	\begin{align}\label{117}
	\lim_{\operatorname{d}_{g_0}(x,\partial M)\rightarrow 0}\, \frac{u(x)}{\mathrm{d}_{g_0}(x,\partial M)}= 1.
	\end{align}
In particular, $g_u$ is complete\footnote{Here and henceforth, when $g_u$ is non-smooth, by completeness of $g_u$ we mean completeness of the metric space on the interior of $M$ induced by $g_u$.} in the interior of $M$. Moreover, if $(1,0,\dots,0)\in\Gamma$, then $u$ is smooth and it is the unique solution in the class of continuous viscosity solutions satisfying $u=0$ on $\partial M$.  
\end{thm}

\begin{rmk}
	For fixed $R_0, i_0, S_0,d_0$ and $\sigma$, there are only finitely many diffeomorphism types of manifolds $M$ with non-empty boundary $\partial M$ admitting metrics $g$ such that $(M,g)\in\mathcal{M}^n_\sigma(R_0, i_0, S_0, d_0)$ -- see Remark \ref{501}. 
\end{rmk}

Our next main result demonstrates that if $M$ is a domain inside a closed Riemannian manifold, then one may remove the dependence of $\delta$ in Theorem \ref{A'} on bounds for $\operatorname{Ric}_{(\partial M,g)}, i_{(\partial M,g)}, i_{(M,g)}^b$ and $\|H\|_{C^\sigma(\partial M,g)}$. 

\begin{defn}\label{502}
	For $R_0\geq 0$ and $i_0,d_0>0$, let $\widetilde{\mathcal{M}}^n(R_0, i_0, d_0)$ denote the set of smooth closed $n$-dimensional Riemannian manifolds $(N,g)$ such that
	\begin{align*}
	-R_0 g \leq \operatorname{Ric}_{(N,g)} \leq R_0 g, \quad i_{(N,g)} \geq i_0 \quad \text{and} \quad \operatorname{diam}(N,g) \leq d_0. 
	\end{align*}
\end{defn}

We prove:

\begin{thm}\label{H''}
	Let $n\geq 3$, $R_0 \geq 0$ and $i_0, d_0>0$. Suppose $(N,g)\in\widetilde{\mathcal{M}}^n(R_0, i_0, d_0)$ and let $\Omega$ be a non-empty open subset of $N$. Then there exists $\ep = \ep(n, N, \Omega, R_0, i_0, d_0)>0$ such that \eqref{105} admits a maximal locally Lipschitz viscosity solution $g_u = u^{-2}g_0$ satisfying \eqref{117} whenever $M\subset N\backslash \Omega$ is the closure of a domain with smooth non-empty boundary and $(f,\Gamma)$ satisfies \eqref{21'}--\eqref{25'} and $\mu_\Gamma^+>1-\ep$. In particular, $g_u$ is complete in the interior of $M$. Moreover, if $(1,0,\dots,0)\in\Gamma$, then $u$ is smooth and it is the unique solution in the class of continuous viscosity solutions satisfying $u=0$ on $\partial M$. 
\end{thm}

Let us now discuss some previous results on the fully nonlinear Loewner-Nirenberg problem to put Theorems \ref{A'} and \ref{H''} into context. The $\sigma_k$-Loewner-Nirenberg problem was first studied for $k\geq 2$ by Mazzeo \& Pacard in \cite{MP03}, where they established perturbative existence results and studied links with the existence of Poincar\'e-Einstein metrics. The first general existence result for the fully nonlinear Loewner-Nirenberg problem was obtained by Gonz\'alez, Li \& Nguyen in \cite{GLN18}, where they showed (among other results) that any smooth bounded Euclidean domain $\Omega$ admits a complete conformally flat locally Lipschitz metric solving \eqref{105} in the viscosity sense, and moreover the solution is unique in the class of continuous viscosity solutions. In later work of Li \& Nguyen \cite{LN20b} and Li, Nguyen \& Xiong \cite{LNX22}, it was shown that the solution of \cite{GLN18} is not $C^1_{\operatorname{loc}}$ if $\Gamma\subseteq \Gamma_2^+$ and $\partial\Omega$ contains more than one connected component; this is in stark contrast to the case $f=c\sigma_1$, where solutions are always smooth in the interior. Recently, Wu \cite{Wu24} showed that solutions to the $\sigma_k$-Loewner-Nirenberg problem are smooth away from a set of measure zero if $k=n$, or if $k>\frac{n}{2}$ and $g_0$ is Euclidean.\medskip 

In the Riemannian setting, Guan \cite{Guan08} proved existence of smooth solutions to a related problem in which the Schouten tensor is replaced by the Ricci tensor, under the assumption that there exists a smooth metric $g\in[g_0]$ satisfying $\lambda(-g^{-1}\operatorname{Ric}_g)\in\Gamma$ on $M$. Such equations involving the Ricci tensor can be recast as special instances of \eqref{105} in which $(1,0,\dots,0)\in\Gamma$ and $\mu_\Gamma^+ \leq 1$ (see Remark \ref{6}). Guan's existence results were subsequently extended by Yuan \cite{Yuan22}, who proved the existence of a smooth complete metric $g_w = w^{-2}g_0$ solving \eqref{105}, assuming $(1,0,\dots,0)\in\Gamma$ and the existence of $g\in[g_0]$ satisfying $\lambda(-g^{-1}A_{g})\in\Gamma$. In \cite{Yuan22}, Yuan also proves uniqueness of solutions and the asymptotics in \eqref{300} under a further technical assumption on $\Gamma$. We point out that the condition $(1,0,\dots,0)\in\Gamma$ fails in many cases of interest, in particular the $\sigma_k$-Loewner-Nirenberg problem when $k\geq 2$.\medskip

To gain a more complete understanding of the fully nonlinear Loewner-Nirenberg problem, it is therefore of significant interest to determine whether the aforementioned assumptions on conformal flatness, the cone $\Gamma$ and/or the existence of certain metrics $g\in[g_0]$ can be removed. In this direction, Gursky, Streets \& Warren \cite{GSW11} obtained the same existence result of Guan \cite{Guan08} without assuming the existence of $g\in[g_0]$ satisfying $\lambda(-g^{-1}\operatorname{Ric}_g)\in\Gamma$; they also proved uniqueness and asymptotics of solutions. In \cite{DN23}, the present authors proved the existence of a maximal Lipschitz viscosity solution to the fully nonlinear Loewner-Nirenberg problem provided $\mu_\Gamma^+>1$, without any assumption on the existence of $g\in[g_0]$ satisfying $\lambda(-g^{-1}A_g)\in\Gamma$. Therein, we also proved asymptotics of solutions, and showed that solutions are smooth and unique (in the class of continuous viscosity solutions) when $(1,0,\dots,0)\in\Gamma$. We point out that some of the estimates in \cite{DN23} do not hold when $\mu_\Gamma^+ \leq 1$ (see our recent work \cite{DN25a} for a further discussion). Using Morse-theoretic methods, Yuan \cite{Yuan24} independently proved the existence of a smooth metric $g\in[g_0]$ satisfying $\lambda(-g^{-1}A_g)\in\Gamma$ in the following three cases: (1) $n=3$, (2) $n\geq 4$ and $\mu_\Gamma^+>1$, and (3) $n\geq 4$, $\mu_\Gamma^+ \geq 1$ and $(1,0,\dots,0)\in\Gamma$.\medskip

We now make some remarks on the proof of Theorem \ref{A'} (the proof of Theorem \ref{H''} follows similar ideas). The approach can be split into two main steps:
\begin{enumerate}
	\item First, we show that there exists a constant $\delta = \delta(n, R_0, i_0, S_0, d_0, \sigma)>0$ such that whenever $(M,g_0)$ belongs to $\mathcal{M}_\sigma^n(R_0, i_0, S_0, d_0)$ and $\Gamma$ satisfies \eqref{21'}, \eqref{22'} and $\mu_\Gamma^+ > 1-\delta$, there exists a smooth metric $g\in[g_0]$ satisfying $\lambda(-g^{-1}A_g)\in\Gamma$.
	
	\item Second, we show that the fully nonlinear Loewner-Nirenberg problem admits a maximal locally Lipschitz viscosity solution (which is smooth and the unique continuous viscosity solution when $(1,0,\dots,0)\in\Gamma$) satisfying \eqref{117} whenever there exists $g\in[g_0]$ satisfying $\lambda(-g^{-1}A_g)\in\Gamma$ on $M$. 
\end{enumerate} 

The approach in Step 1 is to use a global auxiliary function to obtain a metric $g\in[g_0]$ satisfying $\lambda(-g^{-1}A_g)\in\Gamma$. The idea of constructing a suitable auxiliary function to obtain an admissible metric was used by Gursky, Streets \& Warren in \cite{GSW11}, who gave an explicit construction involving the distance function to a point outside the manifold. We instead borrow ideas from Yuan \cite{Yuan24}, who used a non-explicit auxiliary Morse function to construct a smooth metric $g\in[g_0]$ satisfying $\lambda(-g^{-1}A_g)\in\Gamma$ when $\mu_\Gamma^+>1$, or when $\mu_\Gamma^+ \geq 1$ under the additional assumption $(1,0,\dots,0)\in\Gamma$. We show that Yuan's choice of $g$ in fact satisfies $\lambda(-g^{-1}A_g) \in\Gamma$ whenever $\mu_\Gamma^+ \geq 1$, without any further assumptions on $\Gamma$. We then appeal to the $C^{\ell,\sigma}$-compactness theory on compact manifolds with boundary under suitable curvature bounds, following the methods of Anderson et.~al.~\cite{AKKLT04}, to obtain the improvement $\mu_\Gamma^+ > 1-\delta(R_0, i_0, S_0, d_0, \sigma)$. For later reference, we state the main result in Step 1 as a separate result:

\begin{thm}\label{B'}
	Let $n\geq 3$, $R_0, S_0 \geq 0$, $i_0, d_0>0$ and $\sigma\in(0,1)$. There exists a constant $\delta = \delta(n, R_0, i_0, S_0, d_0, \sigma)>0$ such that the following holds. Suppose $(M,g_0)\in\mathcal{M}_\sigma^n(R_0, i_0, S_0, d_0)$ and $\Gamma$ satisfies \eqref{21'}, \eqref{22'} and $\mu_\Gamma^+>1-\delta$. Then there exists a smooth metric $g\in[g_0]$ such that $\lambda(-g^{-1}A_g)\in\Gamma$ on $M$. 
\end{thm}

For Step 2, we prove the following slightly more general theorem which does not require $g$ to be smooth. This provides a unified result that yields locally Lipschitz viscosity solutions when $(1,0,\dots,0)\in\partial\Gamma$, unique smooth solutions when $(1,0,\dots,0)\in\Gamma$, \textit{and} asymptotics in either case:

\begin{thm}\label{C}
	Suppose that $(f,\Gamma)$ satisfies \eqref{21'}--\eqref{25'}, and let $(M,g_0)$ be a smooth compact Riemannian manifold of dimension $n\geq 3$ with smooth non-empty boundary $\partial M$. Suppose that there exists a continuous metric $g_w = w^{-2}g_0$ satisfying 
	\begin{align*}
	f(\lambda(-g_w^{-1}A_{g_w})) \geq \ep>0  , \quad \lambda(-g_w^{-1}A_{g_w})\in\Gamma 
	\end{align*}
	in the viscosity sense on $M$ for some $\ep>0$. Then \eqref{105} admits a maximal locally Lipschitz viscosity solution $g_u = u^{-2}g_0$ satisfying \eqref{117}. In particular, $g_u$ is complete in the interior of $M$. Moreover, if $(1,0,\dots,0)\in\Gamma$, then $u$ is smooth and it is the unique solution in the class of continuous viscosity solutions satisfying $u=0$ on $\partial M$. 
\end{thm}

\begin{proof}[Proof of Theorem \ref{A'}]
	This follows immediately from Theorem \ref{B'} and Theorem \ref{C}.
\end{proof}

\begin{rmk}
	Uniqueness of the Lipschitz viscosity solution to \eqref{105} is also known to hold when $(1,0,\dots,0)\in\partial\Gamma$ on Euclidean domains -- this is due to Gonz\'alez, Li \& Nguyen \cite{GLN18}, who used a comparison principle established by Li, Nguyen \& Wang in \cite{LNW18}. It is not currently known whether such a comparison principle holds in the non-Euclidean setting when $(1,0,\dots,0)\in\partial\Gamma$, and thus uniqueness remains an open question in these cases.
\end{rmk}

\begin{rmk}
	As an immediate consequence of Theorem \ref{C} and the aforementioned existence result of Yuan \cite{Yuan24} when $n=3$, the fully nonlinear Loewner-Nirenberg problem is now solved when $n=3$. 
\end{rmk}

In the proof of Theorem \ref{C}, the existence of a smooth solution to the fully nonlinear Loewner-Nirenberg problem when $(1,0,\dots,0)\in\Gamma$ is achieved by first obtaining for each constant $\delta>0$ a solution $u_\delta$ to \eqref{105} satisfying $u_\delta=\delta$ on $\partial M$, and then taking a limit as $\delta\rightarrow 0$. Our locally Lipschitz viscosity solution to the fully nonlinear Loewner-Nirenberg problem when $(1,0,\dots,0)\in\partial\Gamma$ is then obtained as a limit of solutions to the fully nonlinear Loewner-Nirenberg problem on cones satisfying $(1,0,\dots,0)\in\Gamma$. For this purpose, it is important to have interior $C^1$ estimates (depending on two-sided $C^0$ bounds) which are stable under perturbations of $(f,\Gamma)$ and include both the cases $(1,0,\dots,0)\in\Gamma$ and $(1,0,\dots,0)\in\partial\Gamma$. Such an estimate has recently been established in a general setting by Chu, Li \& Li \cite{CLL23}. The validity of such a local gradient estimate also allows us to apply our estimates near the boundary in \cite[Section 4]{DN23} verbatim, yielding the desired asymptotics of solutions. We refer the reader to Section \ref{s100} for the details.\medskip

As indicated in the previous paragraph, the existence of a locally Lipschitz viscosity solution to the fully nonlinear Loewner-Nirenberg problem when $(1,0,\dots,0)\in\partial\Gamma$ does not require the existence of solutions to \eqref{105} with positive boundary data when $(1,0,\dots,0)\in\partial\Gamma$. However, the existence of solutions to \eqref{105} with positive boundary data when $(1,0,\dots,0)\in\partial\Gamma$ is of independent interest, and requires us to establish a new boundary gradient estimate (see Proposition \ref{123}). The existence result is as follows:

\begin{thm}\label{F}
	Assume the same hypotheses as in Theorem \ref{C} and let $\phi \in C^\infty(\partial M)$ be positive. Then \eqref{105} admits a Lipschitz viscosity solution $g_u = u^{-2}g_0$ satisfying $u=\phi$ on $\partial M$. Moreover, if $(1,0,\dots,0)\in\Gamma$, then $u$ is smooth and it is the unique solution in the class of continuous viscosity solutions satisfying $u=\phi$ on $\partial M$.
\end{thm}

The existence of solutions to the fully nonlinear Loewner-Nirenberg problem on general compact Riemannian manifolds with non-empty boundary remains an interesting open problem for values of $\mu_\Gamma^+$ closer to zero when $n\geq 4$. \medskip 

The plan of the paper is as follows. In Section \ref{s2} we prove Theorem \ref{B'}. In Section \ref{s4} we prove Theorems \ref{C} and \ref{F}. In Appendix \ref{appa} we prove Theorem \ref{19}, a result used in Section \ref{s2} concerning $C^{\ell,\sigma}$-compactness of manifolds with non-empty boundary.\bigskip 

\noindent\textbf{Rights retention statement:} For the purpose of Open Access, the authors have applied a CC BY public copyright licence to any Author Accepted Manuscript (AAM) version arising from this submission.

\section{Proof of Theorem \ref{B'}: existence of admissible metrics when $\mu_\Gamma^+>1-\delta$}\label{s2}

The goal of this section is to carry out the procedure outlined in Step 1 of the proof of Theorem \ref{A'} given in the introduction, i.e.~to prove Theorem \ref{B'}. It will be useful to introduce the following quantity:
\begin{align}\label{503}
\mu(M,[g_0]) = \inf\big\{\mu \geq 0 : & \text{ for all }\Gamma \text{ satisfiying \eqref{21'}, \eqref{22'} and }\mu_\Gamma^+ > \mu, \text{ there exists a } \nonumber \\
& \text{ smooth metric }g\in[g_0] \text{ satisfying }\lambda(-g^{-1}A_g)\in\Gamma \text{ on }M\big\}. 
\end{align}
It is then clear that Theorem \ref{B'} is equivalent to: 

\begin{customthm}{\ref{B'}$'$}\label{B''}
	\textit{Let $n\geq 3$, $R_0, S_0 \geq 0$, $i_0, d_0>0$ and $\sigma\in(0,1)$. Then there exists a constant $\delta = \delta(n,R_0, i_0, S_0, d_0, \sigma)>0$ such that $\mu(M,[g_0]) \leq 1-\delta$ whenever $(M,g_0)\in\mathcal{M}_\sigma^n(R_0, i_0, S_0, d_0)$.}
\end{customthm}

\subsection{A key lemma and admissible metrics when $\mu_\Gamma^+ \geq 1$}

In this subsection, we prove that $\mu(M,[g]) \leq 1 -\eta<1$ for an arbitrary smooth compact Riemannian manifold $(M,g)$ with non-empty boundary, where $\eta>0$ depends on an upper bound for the Schouten tensor of $g$ and the $C^1$ norm of the metric and inverse metric components with respect to a given open covering of charts on $M$. This is the content of the following lemma, which is a key tool in the proof of Theorem \ref{B'}.

\begin{lem}\label{8}
	Let $M$ be a smooth compact manifold of dimension $n\geq 3$ with smooth non-empty boundary $\partial M$ and let $\mathcal{U}$ be any smooth finite open covering of charts on $M$. Then for any $C_1, C_2 > 0$, there exists $\eta = \eta(M, \mathcal{U}, C_1, C_2)>0$ such that if $g$ is any Riemannian metric on $M$ with components $g_{ab}$ in local coordinates determined by $\mathcal{U}$ satisfying 
	\begin{align}\label{10}
 |g_{ab}| +  |g^{ab}| + |\partial_a g_{bc}|  \leq C_1
	\end{align}
	and 
	\begin{align}\label{10'}
	 A_g \leq C_2 g,
	\end{align}
	then $\mu(M,[g])\leq 1-\eta$.
\end{lem}
 
Before giving the proof of Lemma \ref{8}, we note that it has the following immediate consequence: 
\begin{cor}
	Let $(M,g_0)$ be a smooth compact Riemannian manifold of dimension $n\geq 3$ with non-empty boundary. Then $\mu(M,[g_0])<1$. In particular, if $\Gamma$ satisfies \eqref{21'}, \eqref{22'} and $\mu_\Gamma^+ \geq 1$, then there exists a smooth metric $g\in[g_0]$ such that $\lambda(-g^{-1}A_g)\in\Gamma$ on $M$. 
\end{cor}

We recall that for $\mu_\Gamma^+>1$, the last assertion in the corollary was obtained independently by the authors \cite{DN23} and Yuan \cite{Yuan24}. Our proof of Lemma \ref{8} is inspired by the approach of Yuan. In the course of the proof, we make use of a good term that was dropped in Yuan's argument -- see the sentence following \eqref{4} below.

\begin{proof}[Proof of Lemma \ref{8}]
We start by fixing a smooth function $v\geq 1$ on $M$ with no critical points, so that $dv\not=0$ on $M$ -- note that such a function exists since $M$ is a smooth compact manifold with non-empty boundary. In what follows, our constants may depend on $v$, although such dependence is implicit through dependence on $M$. From this point onwards, all computations are carried out in a single chart where $\delta$ is used as a local reference metric, and $|X|_\delta$ denotes the norm of a vector field $X$ with respect to $\delta$. Note that by \eqref{10}, there exists a constant $C_0\geq 1$ depending only on $n$ and $C_1$ such that 
\begin{align}\label{51}
C_0^{-1}(\delta_{ab}) \leq (g_{ab}) \leq C_0 (\delta_{ab})
\end{align}
in the sense of matrices. Note that \eqref{51} is equivalent to $C_0^{-1}|X|_\delta^2 \leq |X|_g^2 \leq C_0|X|_\delta^2$ for all vector fields $X$. Additionally, after enlarging $C_0$ if necessary, the Christoffel symbols $\Gamma_{ab}^c$ of the Levi-Civita connection for $g$ satisfy
\begin{align}\label{418}
|\Gamma_{ab}^c| \leq C_0, 
\end{align}
which follows from \eqref{10} and the fact that $\Gamma_{ab}^c$ can be expressed in terms of the inverse metric components and first order partial derivatives of the metric components. \medskip 

For $N>0$, denote $g ^N = e^{2e^{Nv}}g$. We will show that there exists $N = N(M, \mathcal{U},C_1,C_2)$ such that if $\mu_\Gamma^+ \geq 1-\frac{1}{e^{N\max_M v}}$, then $\lambda(-g^{-1}A_{g^{N}})\in\Gamma$. To this end, we first observe 
\begin{align}\label{415}
-g ^{-1}A_{g ^N} = Ne^{Nv} g ^{-1}\nabla_{g }^2 v + N^2 e^{Nv}(1-e^{Nv}) g ^{-1}(dv\otimes dv) + \frac{1}{2}N^2 e^{2Nv}|dv|_{g }^2 \operatorname{Id} - g ^{-1}A_{g }. 
\end{align}
Now, 
\begin{align}\label{3}
\lambda\bigg(& N^2 e^{Nv}(1-e^{Nv}) g ^{-1}(dv\otimes dv) + \frac{1}{2}N^2 e^{2Nv}|dv|_{g }^2 \operatorname{Id}\bigg) \nonumber \\
& = \frac{1}{2}N^2 e^{2Nv}|dv|_{g }^2\bigg(-1+\frac{2}{e^{Nv}} \, , \, 1 \, , \,  \dots \, , \, 1\bigg)
\end{align}
and (interpreting inequalities of vectors componentwise)
\begin{align}\label{4}
\lambda(-g ^{-1}A_{g }) & \geq -(C_2,\dots,C_2) \nonumber \\
&  = -\frac{1}{2}N^2 e^{2Nv} |dv|_{g }^2\bigg(\frac{2C_2}{N^2e^{2Nv}|dv|_{g }^2}, \dots, \frac{2C_2}{N^2e^{2Nv}|dv|_{g }^2}\bigg) \nonumber \\
& \geq -\frac{1}{2}N^2 e^{2Nv} |dv|_{g }^2\bigg(\frac{2C_0C_2}{N^2e^{2Nv}|dv|_{\delta}^2}, \dots, \frac{2C_0C_2}{N^2e^{2Nv}|dv|_{\delta}^2}\bigg).
\end{align}
We point out that the term $\frac{2}{e^{Nv}}$ in the first component on the bottom line of \eqref{3} was dropped at this point in Yuan's argument \cite{Yuan24}.\medskip 

We next consider $Ne^{Nv}\lambda(g ^{-1}\nabla_{g }^2 v)$, which is the only remaining term coming from \eqref{415}. Recalling
\begin{align*}
(\nabla_{g }^2 v)_{ab}  = \partial_a \partial_b v - \Gamma_{ab}^c\,\partial_c v
\end{align*}
and \eqref{418}, we see that there exists a constant $C_3=C_3(M, \mathcal{U}, C_1)>0$ such that 
\begin{align}\label{5}
Ne^{Nv}\lambda(g ^{-1}\nabla_{g }^2 v)&  \geq -Ne^{Nv}(C_3,\dots, C_3) \nonumber \\
& = -\frac{1}{2}N^2 e^{2Nv}|dv|_{g }^2\bigg(\frac{2C_3}{Ne^{Nv}|dv|_{g }^2}, \dots, \frac{2C_3}{Ne^{Nv}|dv|_{g }^2}\bigg) \nonumber \\
& \geq -\frac{1}{2}N^2 e^{2Nv}|dv|_{g }^2\bigg(\frac{2C_0C_3}{Ne^{Nv}|dv|_{\delta}^2}, \dots, \frac{2C_0C_3}{Ne^{Nv}|dv|_{\delta}^2}\bigg). 
\end{align}

By \eqref{415}--\eqref{5}, we therefore have
\begin{align*}
\lambda(-g ^{-1}A_{g^N}) \geq \frac{1}{2}N^2 e^{2Nv}|dv|_{g }^2 (\chi_1, \chi_2, \dots, \chi_2)
\end{align*}
where
\begin{align*}
\chi_2 & = 1 - \frac{2C_0C_2}{N^2e^{2Nv}|dv|_{\delta}^2} - \frac{2C_0C_3}{Ne^{Nv}|dv|_{\delta}^2}
\end{align*} 
and
\begin{align*}
\chi_1 & = -1+\frac{2}{e^{Nv}} - \frac{2C_0C_2}{N^2e^{2Nv}|dv|_{\delta}^2} - \frac{2C_0C_3}{Ne^{Nv}|dv|_{\delta}^2} \nonumber \\
& = -\chi_2 + \frac{2}{e^{Nv}} - \frac{4C_0C_2}{N^2e^{2Nv}|dv|_{\delta}^2} - \frac{4C_0C_3}{Ne^{Nv}|dv|_{\delta}^2}.
\end{align*}
Now observe that we may fix $N = N(M, \mathcal{U}, C_1, C_2)$ such that $\frac{1}{2}<\chi_2 <1$ and $\chi_1 > -\chi_2 + \frac{1}{e^{Nv}}$. Therefore, if $\mu_\Gamma^+ \geq 1-\frac{1}{e^{N\max_M v}}$, we have $\chi_1 > -\mu_\Gamma^+ \chi_2$ and hence $\lambda(-g^{-1}A_{g^{N}})\in\Gamma$ on $M$.
\end{proof}

\subsection{Tools concerning $C^{\ell,\sigma}$-compactness of Riemannian manifolds with boundary}\label{42}

In this subsection we introduce two theorems that will be used in the proof of Theorem \ref{B'} alongside Lemma \ref{8}. The first result (Theorem \ref{19}) is a minor variation of a theorem of Anderson et.~al.~\cite{AKKLT04} concerning the harmonic radius under suitable curvature bounds on compact manifolds with non-empty boundary. The second result (Theorem \ref{20}) concerns pre-compactness in the $C^{\ell,\sigma}$-topology.\medskip 

We begin by introducing the space $\mathcal{N}^n(\ell+\sigma,\rho,Q)$ for manifolds with boundary:

\begin{defn}\label{17}
	Let $n\geq 3$, $\ell\in\mathbb{N}\cup\{0\}$, $\sigma\in(0,1)$, $\rho\in(0,\infty)$ and $Q>1$. We define $\mathcal{N}^n(\ell+\sigma,\rho, Q)$ to be the set of $n$-dimensional Riemannian manifolds $(M, g)$ with non-empty boundary $\partial M$, where $M$ is smooth and $g$ is of class $C^{\ell,\sigma}$, such that for each $p\in M$: 
\begin{enumerate}
	\item If $d_g(p,\partial M)\geq \frac{3\rho}{4}$, there exists a neighbourhood $U$ of $p$ contained in the interior of $M$ and a chart $\phi: U\rightarrow \mathbb{R}^n$ with $\phi(U) = B_{\frac{3\rho}{4Q}}(0)$ and $\phi(p)=0$ such that
	\begin{align}\label{14}
	Q^{-2}(\delta_{ij}) \leq (g_{ij}) \leq Q^2(\delta_{ij}) \,\,\text{ in }U \quad \text{and} \quad \rho^{\ell+\sigma}\sum_{|\beta| = \ell}[\partial^\beta g_{ij}]_{C^{0,\sigma}(U)} \leq  Q-1,
	\end{align}
	where
	\begin{align}
	[\partial^\beta g_{ij}]_{C^{0,\sigma}(U)} = \sup_{x\not=y\in U}\frac{|\partial^\beta g_{ij}(x) - \partial^\beta g_{ij}(y)|}{|x-y|^\sigma}
	\end{align}
	is the $C^{0,\sigma}$-seminorm of $\partial^\beta g_{ij}$ measured with respect to the coordinates induced by $\phi$, and $|x-y|$ is the Euclidean distance between $x$ and $y$. 
	\item If $d_g(p,\partial M)\leq \rho$, there exists a unique $q\in\partial M$ such that $d_g(p,q) = d_g(p,\partial M)$, and there exist a neighbourhood $U$ of $p$ in $M$ and a chart $\phi:U\rightarrow \overline{\mathbb{R}_+^n}$ with $\phi(U) =  B_{2Q\rho}^+(0) \defeq \{x_n\geq 0\}\cap B_{2Q\rho}(0)$, $\phi(\partial M\cap U) = \{x_n=0\}\cap B_{2Q\rho}(0)$ and $\phi(q) = 0$ such that \eqref{14} holds. 
\end{enumerate}
We further define:
\begin{enumerate}[$\bullet$]
	\item $\mathcal{N}^n_{\mathrm{har}}(\ell+\sigma,\rho, Q)\subset \mathcal{N}^n(\ell+\sigma,\rho, Q)$ by imposing the additional restriction that the coordinate functions in the above two statements are harmonic (i.e.~$\Delta_g \phi^i=0$ for each $i$).
	\item $\mathcal{N}^n(\ell+\sigma,\rho,Q,d_0)\subset \mathcal{N}^n(\ell+\sigma,\rho,Q)$ by imposing the additional restriction $\operatorname{diam}(M,g)\leq d_0$.
\end{enumerate} 
\end{defn}

The following theorem concerns the relationship between the spaces $\mathcal{M}^n_\sigma(R_0, i_0, S_0,\infty)$ (see Definition \ref{500} in the introduction) and $\mathcal{N}_{\mathrm{har}}^n(1+\sigma, \varrho, Q)$ for suitable $\varrho$. It roughly corresponds to \cite[Theorem 3.2.1]{AKKLT04} (more precisely, it is the counterpart to  \cite[Equation (3.1.3)]{AKKLT04}), but requires slightly less regularity: 

\begin{thm}\label{19}
	\textit{Let $n\geq 3$. Given $R_0, S_0 \geq 0$, $i_0>0$, $\sigma\in(0,1)$ and $Q>1$, there exists $\varrho = \varrho(n,R_0, i_0, S_0, Q, \sigma)>0$ such that}
	\begin{align*}
	\mathcal{M}^n_\sigma(R_0, i_0, S_0,\infty) \subset \mathcal{N}_{\mathrm{har}}^n(1+\sigma, \varrho, Q). 
	\end{align*}
\end{thm}

The only difference between Theorem \ref{19} and the analogous result \cite[Equation (3.1.3)]{AKKLT04} is that we assume a H\"older bound on the mean curvature rather than a Lipschitz bound -- as a consequence, we obtain a slightly weaker compactness statement which suffices for our purposes. The proof of Theorem \ref{19} is along the lines of \cite{AKKLT04}, and is given in Appendix \ref{appa} rather than the main body of the paper to improve the exposition. \medskip

The next result needed in the proof of Theorem \ref{B''} concerns the following notion of $C^{\ell,\sigma}$-convergence: 
\begin{defn}\label{401}
	A sequence of Riemannian manifolds $\{(M_k,g_k)\}$ converges in the $C^{\ell,\sigma}$-topology to $(M,g)$ if $M$ is a smooth manifold, $g$ is a $C^{\ell,\sigma}$ Riemannian metric on $M$, and there exist a sequence of diffeomorphisms $F_k:M\rightarrow M_k$ for $k$ sufficiently large and a locally finite collection of charts on $M$ with respect to which $(F_k^*g_k)_{ij}\rightarrow g_{ij}$ in the $C^{\ell,\sigma}$-topology. 
\end{defn}

\begin{thm}\label{20}
	Let $n\geq 3$, $\ell\in\mathbb{N}\cup\{0\}$, $0<\sigma'<\sigma<1$, $\rho\in(0,\infty)$ and $Q>1$. Then for each $d_0>0$, $\mathcal{N}^n(\ell+\sigma,\rho,Q,d_0)$ is precompact in the $C^{\ell,\sigma'}$-topology. 
\end{thm}
\begin{proof}
	A detailed proof in the case of manifolds without boundary may be found in \cite[Theorem 11.3.6]{Pet16}. As observed in \cite{AKKLT04}, the proof extends to the case of non-empty boundary simply by allowing for charts of type 2 as introduced in Definition \ref{17}. 
\end{proof}

\subsection{Proof of Theorem \ref{B'}}

We now give the proof of Theorem \ref{B''} (equivalently Theorem \ref{B'}):

\begin{proof}[Proof of Theorem \ref{B''}]
	Recall the definition of $\mu(M,[g])$ from \eqref{503} and suppose for a contradiction that no such $\delta = \delta(n,R_0, i_0, S_0, d_0, \sigma)>0$ exists. By \cite{DN23, Yuan24}, $\mu(M,[g]) \leq 1$ and hence there exists a sequence of Riemannian manifolds $\{(M_i, g_i)\}\subset \mathcal{M}_\sigma^n(R_0, i_0, S_0, d_0)$ such that $\mu(M_i,[g_i])\rightarrow 1^-$. Let $\varrho = \varrho(n, R_0, i_0, S_0, \frac{3}{2},\sigma)>0$ be as in Theorem \ref{19} with $Q = \frac{3}{2}$. Then $(M_i,g_i)\in\mathcal{N}^n(1+\sigma,\frac{\varrho}{2},\frac{3}{2},d_0)$ for each $i$ and hence, by Theorem \ref{20}, there exists a subsequence (still indexed by $i$), a smooth manifold $M_\infty$ and a $C^{1,\sigma}$ Riemannian metric $g_\infty$ on $M_\infty$ such that $(M_\infty, g_\infty)\in\mathcal{N}^n(1+\sigma,\frac{\varrho}{2},\frac{3}{2},d_0)$ and $(M_i,g_i)\rightarrow (M_\infty, g_\infty)$ in the $C^{1,\sigma'}$-topology for any $\sigma'<\sigma$. In particular, the $M_i$ are eventually diffeomorphic to $M_\infty$, and so we may assume without loss of generality that $M_i = M_\infty$ for each $i$. Then by definition of $C^{1,\sigma'}$ convergence, there exists a finite open covering $\mathcal{U}_\infty$ with respect to which $g_i\rightarrow g_\infty$ and $g_i^{-1} \rightarrow g_\infty^{-1}$ componentwise in $C^{1,\sigma'}$. Moreover, after a refinement of $\mathcal{U}_\infty$ if necessary (using Theorem \ref{19}), the bounds in Definition \ref{17} are satisfied with respect to $\mathcal{U}_\infty$ with $Q=2$ and $g = g_i$ for any $i<\infty$. Therefore the metrics $g_i$ satisfy the bound in \eqref{10} in local coordinates determined by $\mathcal{U}_\infty$ with $C_1$ depending only on $Q=2$. Moreover, since $(M_i, g_i)\in \mathcal{M}_\sigma^n(R_0, i_0, S_0, d_0)$, each $g_i$ also satisfies the bound \eqref{10'} with $C_2$ depending only on $R_0$ and $n$. Therefore $\mu(M_i,[g_i])<1-\eta(M_\infty, \mathcal{U}_\infty, C_1, C_2)$ independently of $i$, where $\eta$ is the constant in Lemma \ref{8}. This contradicts $\mu(M_i,[g_i])\rightarrow 1$. 
\end{proof}

\begin{rmk}\label{501}
	As a consequence of Theorems \ref{19} and \ref{20}, it is easy to see that given  $R_0, S_0 \geq 0$, $i_0, d_0>0$ and $\sigma\in(0,1)$, there are finitely many diffeomorphism types of compact manifold $M$ with non-empty boundary admitting a Riemannian metric $g$ such that $(M,g)\in\mathcal{M}^n_\sigma(R_0, i_0, S_0,d_0)$.
\end{rmk}

\subsection{A counterpart to Theorem \ref{B'} for domains in closed manifolds}\label{43}

The following result is a counterpart to Theorem \ref{B''} (equivalently Theorem \ref{B'}) which is used in the proof of Theorem \ref{H''} (we refer the reader back to Definition \ref{502} for the definition of $\widetilde{\mathcal{M}}^n(R_0, i_0, d_0)$):

\begin{thm}\label{H}
	Let $n\geq 3$, $R_0 \geq 0$ and $i_0, d_0>0$. Suppose $(N,g)\in\widetilde{\mathcal{M}}^n(R_0, i_0, d_0)$ and $\Omega$ is a non-empty open subset of $N$. Then there exists $\ep = \ep(n, N, \Omega, R_0, i_0, d_0)>0$ such that $\mu(N\backslash\Omega,[g|_{N\backslash \Omega}])<1-\ep$. 
\end{thm}

\begin{proof}[Proof of Theorem \ref{H''}]
	This follows immediately from Theorem \ref{H} and Theorem \ref{C}. 
\end{proof}

For the proof of Theorem \ref{H} we will need the following variant of Lemma \ref{8} for manifolds with boundary arising as domains inside a closed manifold: 

\begin{lem}\label{12}
	Let $N$ be a smooth closed manifold of dimension $n\geq 3$, $\mathcal{U}$ a smooth finite open covering of charts on $N$ and $\Omega$ a non-empty open subset of $N$. Then for any $C_1, C_2>0$, there exists $\eta = \eta(N, \Omega, \mathcal{U}, C_1, C_2)>0$ such that if $g$ is any metric $N$ satisfying \eqref{10} (with respect to local coordinates determined by $\mathcal{U}$) and \eqref{10'}, then $\mu(N\backslash\Omega,[g|_{N\backslash\Omega}])<1-\eta$.
\end{lem}
\begin{proof}
	We start by fixing a smooth function $w \geq 1$ on $N$ with no critical points in $N\backslash \Omega$, so that $dw\not = 0$ on $N\backslash\Omega$ -- note that such a function exists since $N\backslash\Omega$ is a compact manifold with non-empty boundary. The rest of the proof then proceeds as in the proof of Lemma \ref{8}, taking $v$ therein to be equal to $w|_{N\backslash\Omega}$. 
\end{proof}

We now define $\widetilde{\mathcal{N}}^n(\ell+\sigma,\rho,Q)$, $\widetilde{\mathcal{N}}^n(\ell+\sigma,\rho, Q, d_0)$ and $\widetilde{\mathcal{N}}^n_{\mathrm{har}}(\ell+\sigma,\rho,Q)$ in an analogous manner to the definitions in Section \ref{42}, but in the class of \textit{closed} $n$-dimensional Riemannian manifolds (so that charts of the second type in Definition \ref{17} no longer occur). Then we have the following counterparts to Theorems \ref{19} and \ref{20}, respectively:

\begin{thm}\label{19'}
	Let $n\geq 3$. Given $R_0 \geq 0$, $i_0>0$, $\sigma\in(0,1)$ and $Q>1$, there exists $\varrho = \varrho(n, R_0, i_0, Q)>0$ such that
	\begin{align*}
	\widetilde{\mathcal{M}}^n(R_0, i_0,\infty) \subset \widetilde{\mathcal{N}}_{\mathrm{har}}^n(1+\sigma,\varrho, Q).
	\end{align*}
\end{thm}
\begin{proof}
	The proof of Theorem \ref{19} applies verbatim here with Case 1 being the only possibility. Indeed, bounds for $\operatorname{Ric}_{(\partial M,g)}, i_{(\partial M,g)}, i_{(M,g)}^b$ and $\|H\|_{C^{0,\sigma}(\partial M,g)}$ are not used in Case 1 in the proof of Theorem \ref{19}. 
\end{proof}

\begin{thm}\label{20'}
	Let $n\geq 3$, $\ell\in\mathbb{N}\cup\{0\}$, $0<\sigma'<\sigma<1$, $\rho\in(0,\infty)$ and $Q>1$. Then for each $d_0>0$, $\widetilde{\mathcal{N}}^n(\ell+\sigma,\rho,Q,d_0)$ is precompact in the $C^{\ell,\sigma'}$-topology. 
\end{thm}
\begin{proof}
	As in the proof of Theorem \ref{20}, we refer to \cite[Theorem 11.3.6]{Pet16}.
\end{proof}

\begin{proof}[Proof of Theorem \ref{H}]
	Suppose for a contradiction that no such $\ep = \ep(n,N,\Omega, R_0, i_0, d_0)>0$ exists, and recall from \cite{DN23,Yuan24} that $\mu(N\backslash\Omega,[g|_{N\backslash\Omega}])\leq 1$. It follows that there exists a sequence of Riemannian metrics $g_i$ with $(N,g_i)\in \widetilde{\mathcal{M}}^n(R_0, i_0, d_0)$ such that $\mu(N\backslash\Omega,[g_i|_{N\backslash\Omega}])\rightarrow 1^-$. Let $\varrho = \varrho(n, R_0, i_0, \frac{3}{2})>0$ be as in Theorem \ref{19'} with $Q=\frac{3}{2}$. Then  $(N,g_i)\in\widetilde{\mathcal{N}}^n(1+\sigma,\frac{\varrho}{2},\frac{3}{2},d_0)$ for each $i$ and hence, by Theorem \ref{20'}, there exists a subsequence (still indexed by $i$) and a $C^{1,\sigma}$ Riemannian metric $g_\infty$ on $N$ such that $(N,g_\infty)\in\in\widetilde{\mathcal{N}}^n(1+\sigma,\frac{\varrho}{2},\frac{3}{2},d_0)$ and $(N,g_i)\rightarrow (N,g_\infty)$ in the $C^{1,\sigma'}$-topology for any $\sigma'<\sigma$. Therefore, there exists a smooth finite open covering $\mathcal{U}_\infty$ of $N$ with respect to which $g_i\rightarrow g_\infty$ and $g_i^{-1}\rightarrow g_\infty^{-1}$ componentwise in $C^{1,\sigma'}$. Moreover, after a refinement of $\mathcal{U}_\infty$ if necessary (using Theorem \ref{19'}), the bounds in \eqref{14} are satisfied with $Q = 2$ and $g = g_i$ for any $i<\infty$. Therefore the metrics $g_i$ satisfy the bound in \eqref{10} in local coordinates determined by $\mathcal{U}_\infty$ with $C_1$ depending only on $Q=2$. Moreover, since $(N,g_i)\in\widetilde{\mathcal{M}}^n(R_0, i_0)$, each $g_i$ also satisfies the bound \eqref{10'} with $C_2$ depending only on $R_0$ and $n$. Therefore $\mu(N,[g_i|_{N\backslash\Omega}]) < 1 - \eta(N,\Omega, \mathcal{U}_\infty, C_1, C_2)$ independently of $i$, where $\eta$ is the constant in Lemma \ref{12}. This contradicts $\mu(N\backslash\Omega,[g_i|_{N\backslash\Omega}])\rightarrow 1$. 
\end{proof}

\section{Proof of Theorems \ref{C} and \ref{F}: solutions from admissible metrics}\label{s4}

\subsection{An equivalent formulation of Theorem \ref{C}}

We first give an equivalent formulation of Theorem \ref{C}, which will be more conducive to our method of proof. Given $(f,\Gamma)$ satisfying \eqref{21'}--\eqref{25'}, a number $\tau\in[0,1]$ and a vector $\lambda\in\mathbb{R}^n$, we define
\begin{align*}
\lambda^\tau \defeq \tau\lambda + (1-\tau)\sigma_1(\lambda)e, \quad f^\tau(\lambda) = \frac{1}{\tau + n(1-\tau)}f(\lambda^\tau), \quad  \Gamma^\tau = \{\lambda:\lambda^\tau\in\Gamma\}. 
\end{align*}
It is then routine to verify that $(f^\tau, \Gamma^\tau)$ also satisfies \eqref{21'}--\eqref{25'}. Moreover, it is known (see \cite[Appendix A]{DN22}) that $\Gamma$ satisfies \eqref{21'}, \eqref{22'} and $(1,0,\dots,0)\in\Gamma$ if and only if there exists a cone $\widetilde{\Gamma}$ satisfying \eqref{21'} and \eqref{22'}, and a number $\tau<1$ such that $\Gamma = (\widetilde{\Gamma})^\tau$. We are therefore led to consider the equation
\begin{align}\label{105'}
\begin{cases}
f^\tau(\lambda(-g_u^{-1}A_{g_u})) = \frac{1}{2}, \quad \lambda(-g_u^{-1}A_{g_u})\in\Gamma^\tau & \text{on }M\backslash \partial M \\
u=0 & \text{on }\partial M.
\end{cases}
\end{align}
An equivalent formulation of Theorem \ref{C} is as follows: 

\begin{customthm}{\ref{C}$'$}\label{C'}
	\textit{Suppose that $(f,\Gamma)$ satisfies \eqref{21'}--\eqref{25'} and $(1,0,\dots,0)\in\partial\Gamma$, and let $(M,g_0)$ be a smooth compact Riemannian manifold of dimension $n\geq 3$ with non-empty smooth boundary $\partial M$. Suppose for some $\tau_0 \leq 1$ and constant $\ep>0$ that there exists a continuous metric $g_w = w^{-2}g_0$ satisfying}
	\begin{align*}
	f^{\tau_0}(\lambda(-g_w^{-1}A_{g_w})) \geq \ep > 0,  \quad \lambda(-g_w^{-1}A_{g_w})\in\Gamma^{\tau_0}
	\end{align*}
	\textit{in the viscosity sense on $M$. Then for each $\tau \leq \tau_0$, \eqref{105'} admits a maximal locally Lipschitz viscosity solution $g_{u_\tau} = u_\tau^{-2}g_0$ satisfying}
	\begin{align*}
	\lim_{\operatorname{d}_{g_0}(x,\partial M)\rightarrow 0} \operatorname{d}_{g_0}(x,\partial M)u_{\tau}^{-1}(x) = 1. 
	\end{align*}
	\textit{Moreover, if $\tau<1$ then $u_\tau$ is smooth and it is the unique solution in the class of continuous viscosity solutions satisfying $u_\tau=0$ on $\partial M$. }
\end{customthm}

\begin{rmk}
	It is known that \eqref{105'} exhibits better ellipticity properties when $\tau<1$, see e.g.~\cite{GV03b, DN22}. From this point of view, \eqref{105'} for $\tau<1$ can be viewed as an elliptic regularisation of \eqref{105'} for $\tau=1$. 
\end{rmk}

\begin{rmk}\label{6}
	It is routine to verify that for $\tau = \frac{n-2}{n-1}$, one has $(\sigma_k^{1/k})^\tau(\lambda(-g^{-1}A_g))  = \frac{1}{n-1}\sigma_k^{1/k}(\lambda(-g^{-1}\operatorname{Ric}_g))$. Therefore, the equations mentioned in the introduction that were considered by Guan \cite{Guan08} and Gursky, Streets \& Warren \cite{GSW11} fall within the scope of \eqref{105}. 
\end{rmk}

Following the outline given in the introduction, in the proof of Theorem \ref{C'} we will first consider the problem \eqref{105'} for $\tau \leq \tau_0$ but with positive constant boundary data $\delta>0$:
\begin{align}\label{105''}
\begin{cases}
f^\tau(\lambda(-g_u^{-1}A_{g_u})) = \frac{1}{2}, \quad \lambda(-g_u^{-1}A_{g_u})\in\Gamma^\tau & \text{on }M\backslash \partial M \\
u=\delta>0 & \text{on }\partial M.
\end{cases}
\end{align}
If $\tau_0<1$, the smooth solution to \eqref{105'} will then be obtained as the limit of solutions to \eqref{105''} as $\delta\rightarrow 0$. If $\tau_0=1$, we obtain smooth solutions for each $\tau<1$ with zero boundary data in the same way, and a locally Lipschitz viscosity solution to \eqref{105'} with $\tau=1$ is then obtained as a (subsequential) limit of these solutions. Although not needed in the proof of Theorem \ref{C'}, it is also of independent interest to establish the existence of a Lipschitz viscosity solution to \eqref{105''} when $\tau=1$. We will therefore prove the following preliminary existence result concerning \eqref{105''}, which is equivalent to Theorem \ref{F} (note that although \eqref{105''} concerns only constant boundary data, a conformal change of background metric can be used to transform any Dirichlet boundary value problem with positive boundary data into one with positive constant boundary data):

\begin{customthm}{\ref{F}$'$}\label{C''}
	\textit{Assume the same set-up as in Theorem \ref{C'}. Then for each $\tau\leq \tau_0$, \eqref{105''} admits a locally Lipschitz viscosity solution. Moreover, if $\tau<1$ then the solution is smooth and it is the unique solution in the class of continuous viscosity solutions.}
\end{customthm}

The existence of a smooth solution in Theorem \ref{C''} when $\tau<1$ will be obtained using the continuity method with $\tau$ as the parameter. The key is to obtain \textit{a priori} $C^1$ estimates that are uniform with respect to $\tau \in[0,\tau_0]$, followed by \textit{a priori} second derivative estimates that are uniform with respect to $\tau\in[0,\tau_0]$ when $\tau_0 < 1$ and locally uniform with respect to $\tau\in[0,1)$ when $\tau_0 = 1$. In the case that $\tau_0 = 1$, the existence of a Lipschitz viscosity solution will be obtained in the limit as $\tau\rightarrow 1$. In view of the counterexamples to $C^1$ regularity when $\tau=1$ in \cite{LN20b, LNX22}, one expects some degeneration of second derivative estimates as $\tau\rightarrow 1$. We now carry about this procedure.

\subsection{$C^0$ estimates}

In this section we first prove the following global $C^0$ estimate on solutions to \eqref{105''}: 

\begin{prop}[Global $C^0$ estimate]\label{121}
	Assume the same set-up as in Theorem \ref{C'}. Then there exist constants $C_1=C_1(g_0, f, \Gamma, \delta)$ and $C_2 = C_2(\delta, w, \ep)$ (with both $C_1$ and $C_2$ independent of $\tau$) such that any $C^2$ solution $u$ to \eqref{105''} with $\tau \leq \tau_0$ satisfies 
	\begin{align*}
	0< C_1 \leq u \leq C_2 \quad \text{on }M. 
	\end{align*}
\end{prop}

\begin{rmk}\label{417}
	It will be clear from the proof of Proposition \ref{121} that $C_2$ is non-decreasing as a function of $\delta$ and hence bounded from above as $\delta\rightarrow 0$. 
\end{rmk}

We will use the following well-known existence result in the proof of Proposition \ref{121} (and also later in the proof of Theorem \ref{C''}), whose proof we summarise for the convenience of the reader: 

\begin{lem}\label{141}
	Let $(M,g_0)$ be a smooth compact Riemannian manifold of dimension $n\geq 3$ with non-empty smooth boundary $\partial M$. Then for each $\delta>0$, there exists a unique smooth solution $g_v = v^{-2}g_0$ to the Dirichlet boundary value problem 
	\begin{align}\label{150}
	\begin{cases}
	R_{g_v} = -2n(n-1) & \text{on }M\backslash \partial M \\
	v = \delta & \text{on }\partial M. 
	\end{cases}
	\end{align}
\end{lem}
\begin{proof}
	Suppose first that the first eigenvalue of the conformal Laplacian of $g_0$ is nonnegative. Then $g_0$ is conformal to a metric with nonnegative scalar curvature, which we assume without loss of generality to be $g_0$ itself. In this case, the direct method of the calculus of variations applies (since the Yamabe functional is both coercive and strictly convex) to yield a unique solution to \eqref{150}.\medskip 
	
	If the first eigenvalue of the conformal Laplacian of $g_0$ is negative, then $g_0$ is conformal to a metric with negative scalar curvature, which we assume without loss of generality to be $g_0$ itself. Recalling that $R_{g_v} = v\Delta_{g_0}v - \frac{n}{2}|\nabla_{g_0}v|^2 + R_{g_0} v^2$, we consider for $t\in[0,1]$ the path of equations
	\begin{align}\label{151}
	\begin{cases}
	v_t\Delta_{g_0}v_t - \frac{n}{2}|\nabla_{g_0}v_t|^2 + R_{g_0} v_t^2 = -2n(n-1)t + R_{g_0}\delta^2 (1-t) & \text{on }M\backslash \partial M \\
	v_t = \delta & \text{on }\partial M. 
	\end{cases}
	\end{align}
	Since $R_{g_0}<0$, for each fixed $t\in[0,1]$ any solution to \eqref{151} is unique. In particular, $v_0 = \delta$ is the unique solution when $t=0$. Invertibility of the linearised operator also follows from the fact that $R_{g_0}<0$. Existence will therefore follow via the continuity method once $C^0$ estimates are established that are independent of $t$. \medskip 
	
	To this end, let $x_0$ be a maximum point for $v_t$. If $x_0\in\partial M$, then $v_t \leq \delta$. If $x_0\in M$, then $\Delta_{g_0} v_t (x_0) \leq 0$ and $\nabla_{g_0}v_t(x_0) = 0$ and hence
	\begin{align*}
	-2n(n-1)t + R_{g_0}(x_0)\delta^2(1-t) \leq R_{g_0}(x_0) v_t(x_0)^2. 
	\end{align*}
	Since $R_{g_0}<0$, this yields an upper bound for $v_t(x_0)$ which can clearly be made independent of $t$. To obtain the lower bound\footnote{One can avoid the use of \cite{AM88} by an ad-hoc but more elementary construction: in the interior one may compare $v_t$ to a perturbation of a Poincar\'e metric on a small ball, and near the boundary one may do an explicit construction using the distance function to the boundary. We instead appeal to \cite{AM88} for a more concise argument.}, extend $M$ via a collar neighbourhood $N$ such that the solution of Aviles \& McOwen \cite{AM88} to the problem
	\begin{align*}
	\begin{cases}
	R_{g_w} = - 2n(n-1) - \sup_M |R_{g_0}| \delta^2 & \text{on }(M\cup N)\backslash \partial(M\cup N) \\
	w = 0 & \text{on }\partial (M\cup N)
	\end{cases}
	\end{align*}
	satisfies $w \leq \delta$ on $\partial M$. Then the comparison principle implies $v_t \geq w$ on $M$. 
\end{proof}

We now give the proof of the global $C^0$ estimate in Proposition \ref{121}:

\begin{proof}[Proof of Proposition \ref{121}]
	Let $c=c(\delta,w,\epsilon)>0$ be a constant such that $c^2 \geq \frac{1}{2\ep}$ and $cw \geq \delta$ on $\partial M$. Then 
	\begin{align*}
	f^{\tau_0}(\lambda(-g_{cw}^{-1}A_{g_{cw}})) = c^2 f^{\tau_0}(\lambda(-g_w^{-1}A_{g_w})) \geq \frac{1}{2}
	\end{align*}
	in the viscosity sense on $M$. By the comparison principle, it follows that $u \leq cw$ on $M$. \medskip 
	
	To obtain the lower bound for $u$, first observe that by concavity and homogeneity of $f$, and the normalisation $f(1,\dots,1) = 1$, we have
	\begin{align*}
	f(\lambda) \leq f\bigg(\frac{\sigma_1(\lambda)}{n}e\bigg) + \nabla f\bigg(\frac{\sigma_1(\lambda)}{n}e\bigg)\cdot \bigg(\lambda - \frac{\sigma_1(\lambda)}{n}e\bigg) = \frac{f(e)}{n}\sigma_1(\lambda) = \frac{1}{n}\sigma_1(\lambda)
	\end{align*}
	for all $\lambda\in\Gamma$. Therefore the scalar curvature of $g_u$ satisfies $R_{g_u} \leq -2n(n-1)$ on $M$ in the viscosity sense. Letting $v$ denote the smooth solution to \eqref{150}, the comparison principle then implies $u \geq v$ on $M$. 
\end{proof}

We point out that although the global $C^0$ estimate in Proposition \ref{121} will suffice in the proof of Theorem \ref{F} (where the positive boundary data is fixed), in order to obtain a solution to the fully nonlinear Loewner-Nirenberg problem we will also need an interior $C^0$ estimate which remains uniform as $\delta\rightarrow 0$. This is the content of the following proposition: 

\begin{prop}\label{416}
	Assume the same set-up as in Theorem \ref{C'}, let $\delta_0>0$ and $M'\Subset M\backslash\partial M$. Then there exist constants $C_3 = C_3(g_0, f, \Gamma, M')$ and $C_4=C_4(\delta_0, w, \ep)$ (with both $C_3$ and $C_4$ independent of $\tau$) such that any $C^2$ solution $u$ to \eqref{105''} with $\tau \leq \tau_0$ and $0<\delta\leq \delta_0$ satisfies
	\begin{align*}
	0<C_3 \leq u \leq C_4 \quad\text{on }M'. 
	\end{align*}
\end{prop}
\begin{proof}
	By Remark \ref{417}, the constant $C_2$ in Proposition \ref{121} is non-decreasing as a function of $\delta$, and hence the existence of $C_4$ with the claimed dependencies is immediate. On the other hand, let $v$ denote the smooth solution to \eqref{150} with $\delta=0$, which exists by Aviles \& McOwen \cite{AM88}. As observed in the proof of Proposition \ref{121}, $R_{g_u} \leq -2n(n-1)$ on $M$ in the viscosity sense, and thus for any boundary data $\delta>0$ we may apply the comparison principle to assert $u \geq v$ on $M$. This yields a positive lower bound for $u$ on any compact subset of $M\backslash \partial M$, as required. 
\end{proof}

\subsection{Gradient estimates}\label{grad}

We next consider gradient estimates on solutions to \eqref{105''}. There are two lines of work in this direction: gradient estimates depending on two-sided $C^0$ bounds, and the more involved gradient estimates depending only on one-sided $C^0$ bounds. In light of Proposition \ref{121}, it suffices in this paper to work with gradient estimates depending on two-sided $C^0$ bounds.\medskip 

 Global gradient estimates on closed manifolds depending on two-sided $C^0$ bounds were obtained by Gursky \& Viaclovsky in \cite{GV03b} for $(f,\Gamma)= (\sigma_k^{1/k},\Gamma_k^+)$ when $k\geq 2$. Local interior and global gradient estimates depending on two-sided $C^0$ bounds were obtained by Guan \cite{Guan08} for solutions to \eqref{105''} for $(f,\Gamma)$ satisfying \eqref{21'}--\eqref{24'} and $\tau<1$. Recently, Chu, Li \& Li \cite{CLL23} obtained the local interior gradient estimate depending on two-sided $C^0$ bounds for $(f,\Gamma)$ satisfying conditions even more general than \eqref{21'}--\eqref{24'} and for all $\tau\leq 1$.\medskip 

For completeness, let us also briefly mention local interior gradient estimates depending only on one-sided $C^0$ bounds. These were first proved by Khomrutai \cite{Kho09} for $(f,\Gamma) = (\sigma_k^{1/k}, \Gamma_k^+)$ when $k<\frac{n}{2}, k=n-1$ or $k=n$. More recently, Chu, Li \& Li \cite{CLL23} obtained the local interior gradient estimate depending only on a one-sided $C^0$ bound whenever $\mu_\Gamma^+\not=1$, and gave counterexamples to such estimates when $\mu_\Gamma^+= 1$. In \cite{DN23}, the present authors obtained both a local interior and a local boundary gradient estimate depending only on one-sided $C^0$ bounds when $\mu_\Gamma^+>1$ and for all $\tau \leq 1$; see \cite[Theorem 1.8]{DN23} and the proof of \cite[Proposition 3.8]{DN23}, respectively. On the other hand, our recent result \cite[Theorem 1.5]{DN25a} demonstrates the failure of local boundary gradient estimates depending on one-sided $C^0$ bounds which are uniform as $\delta\rightarrow 0$ when $\mu_\Gamma^+ \leq 1$.\medskip

The goal of this section is to prove the following: 

\begin{prop}[Global gradient estimate]\label{123}
	Let $(M,g_0)$ be a smooth compact Riemannian manifold of dimension $n\geq 3$ with non-empty smooth boundary $\partial M$, and suppose $(f,\Gamma)$ satisfies \eqref{21'}--\eqref{24'}. Fix $\tau\in(0,1]$, positive functions $\psi\in C^\infty(M)$ and $\phi\in C^\infty(\partial M)$, and suppose that $u\in C^3(M)$ satisfies 
	\begin{align}\label{55}
	\begin{cases}
	f^\tau(\lambda(-g_u^{-1}A_{g_u})) = \psi, \quad \lambda(-g_u^{-1}A_{g_u})\in\Gamma^\tau & \text{in }M\backslash \partial M \\
	u = \phi & \text{on }\partial M.
	\end{cases}	
	\end{align} 
	Then
	\begin{align*}
	|\nabla_{g_0} u|_{g_0} \leq C \quad \text{in }M, 
	\end{align*}
	where $C$ is a constant depending on $n, f, \Gamma$ and upper bounds for $\|g_0\|_{C^3(M)}, \|\psi\|_{C^1(M)}, \|\phi\|_{C^1(\partial M)}$ and $\|\ln u\|_{C^0(M)}$, but independent of $\tau$. 
\end{prop}

The main new aspect of Proposition \ref{123} is the boundary gradient estimate, as interior gradient estimates follow from the recent work of Chu, Li \& Li \cite{CLL23}. As pointed out above, the proof of local boundary gradient estimates depending on one-sided $C^0$ bounds in \cite{DN23} does not work when $\mu_\Gamma^+ \leq 1$. In this paper we provide a variant of the construction in \cite{DN23} that is valid for any value of $\mu_\Gamma^+$, but depends on a two-sided $C^0$ bound near $\partial M$. The main new ingredient is the following counterpart to \cite[Proposition 3.4]{DN23}, which will serve as an upper barrier for $u$ near $\partial M$.

\begin{prop}\label{4'}
	Suppose $(f,\Gamma)$ satisfies \eqref{21'}--\eqref{25'} and let $g_0$ be a Riemannian metric defined on a neighbourhood of the origin in $\mathbb{R}^n$. Fix constants $0< \delta < m$. Then there exists a constant $\overline{R} = \overline{R}(g_0, m)>0$ such that whenever $0<R \leq \overline{R}(g_0, m)$, $r_1 = R\sqrt{1+\delta}$ and $r_2 = R\sqrt{1+m}$, the metric $g_0$ is well-defined on the annulus $A_{r_1,r_2}$, and the conformal metric $g_v = v^{-2}g_0$ defined by 
	\begin{align}\label{3'}
	v(r) = \frac{r^2 - R^2}{R^2}
	\end{align}
	satisfies
	\begin{align*}
	\begin{cases}
	f(\lambda(-g_v^{-1}A_{g_v})) \geq \frac{1}{2}, \quad \lambda(-g_v^{-1}A_{g_v})\in\Gamma & \text{in }A_{r_1,r_2} \\
	v(x) = \delta & \text{on }\{r=r_1\} \\
	v(x) = m & \text{on }\{r=r_2\}.  
	\end{cases}
	\end{align*}
\end{prop}

\begin{rmk}
	In the case that $g_0$ is Euclidean, $g_v$ is an $R$-dependent rescaling of the hyperbolic metric on the exterior of $B_R$. 
\end{rmk}

\begin{proof} 
	In a fixed normal coordinate system centred at the origin, since $v = v(r)>0$ we have (see e.g.~\cite[Appendix B]{DN23})
	\begin{align*}
	(g_v^{-1}A_{g_v})^p_j & =v^2\bigg(\lambda\delta_j^p + \chi\frac{x^px_j}{r^2}\bigg) + O(r^2)v|v_{rr}| + O(r)v|v_r| + O(1)v^2 \quad \text{as }r\rightarrow 0,
	\end{align*}
	where 
	\begin{equation*}
	\lambda = \frac{v_r}{rv}\bigg(1-\frac{rv_r}{2v}\bigg) \quad\text{and}\quad \chi = \frac{v_{rr}}{v} - \frac{v_r}{vr}
	\end{equation*}
	and the big $O$ terms depend only on $g_0$. Therefore 
	\begin{align}\label{130}
	(-g_v^{-1}A_{g_v})^p_j \geq -v^2\bigg(\lambda\delta_j^p + \chi\frac{x^px_j}{r^2}\bigg) - |\Psi|\delta_j^p
	\end{align}
	in the sense of matrices, where $|\Psi| = O(r^2)v|v_{rr}| + O(r)v|v_r| + O(1)v^2$ as $r\rightarrow 0$. Note that the eigenvalues of the quantity on the RHS of \eqref{130} are given by 
	\begin{align}\label{203}
	-(\chi v^2 + \lambda v^2 + |\Psi| \,\, , \,\, \lambda v^2 + |\Psi| \,\, , \,\, \dots \,\, , \,\, \lambda v^2 + |\Psi|). 
	\end{align}
	
	We now compute $v_r = \frac{2r}{R^2}$, $v_{rr} = \frac{2}{R^2}$ and hence
	\begin{align*}
	\frac{v_r}{vr} = \frac{2}{r^2 - R^2}, \quad \frac{rv_r}{2v} = \frac{r^2}{r^2 - R^2} \quad \text{and} \quad \frac{v_{rr}}{v} = \frac{2}{r^2 - R^2}. 
	\end{align*}
	It follows that 
	\begin{align}\label{200}
	\lambda = -\frac{2R^2}{(r^2 - R^2)^2} \quad \text{and}\quad \chi = 0. 
	\end{align}
	On the other hand, 
	\begin{align}\label{201}
	|\Psi| & \leq Cr^2v|v_{rr}| + Crv|v_r| + Cv^2 \leq \frac{Cr^4}{R^4},
	\end{align}
	where here and below $C$ is a constant depending only on $g_0$. Substituting \eqref{200} and \eqref{201} into \eqref{203}, we see that each eigenvalue of $-g_v^{-1}A_{g_v}$ satisfies
	\begin{align}\label{1'}
	\lambda_i(-g_v^{-1}A_{g_v}) \geq -(\lambda v^2 + |\Psi|) \geq 2R^{-2} - \frac{Cr^4}{R^4}.
	\end{align}
	
	Now let $r_1 = R\sqrt{1+\delta}$ and $r_2 = R\sqrt{1+m}$. It is then clear that $v=\delta$ on $\{r=r_1\}$ and $v = m$ on $\{r=r_2\}$. Moreover, since $m$ is fixed, $\overline{R}$ may be fixed sufficiently small (depending on $g_0$ and $m$) so that $R\leq \overline{R}$ implies $B_{r_2}(0)$ is contained in the region in which $g_0$ is defined. Also, by \eqref{1'} we see that in $A_{r_1,r_2}$ each eigenvalue of $-g_v^{-1}A_{g_v}$ satisfies
	\begin{align*}
	\lambda_i(-g_v^{-1}A_{g_v}) \geq  2R^{-2} - \frac{Cr_2^4}{R^4} =  2R^{-2} - C(1+m)^2.
	\end{align*}
	Therefore, after taking $\overline{R}$ smaller if necessary (but in a way that only depends on $g_0$ and $m$), we have $\lambda_i(-g_v^{-1}A_{g_v})\geq \frac{1}{2}$ in $A_{r_1,r_2}$ for each $i$ whenever $R \leq \overline{R}$. Recalling the monotonicity of $f$ in \eqref{24'} and the normalisation \eqref{25'}, the result then follows. 
\end{proof}

We now give the proof of Proposition \ref{123}, which is a combination of ideas from \cite{CLL23, DN23}:

\begin{proof}[Proof of Proposition \ref{123}]
	After a conformal change of background metric, we may assume without loss of generality that the boundary data $\phi$ is equal to a constant $\delta>0$.\medskip 
	
	Suppose $x_0 \in M$ is a point at which the maximum of $|\nabla_{g_0} u|_{g_0}$ is obtained. If $x_0$ is an interior point, then one may run the same argument as in the proof of \cite[Theorem 7.1]{CLL23}, but instead taking $\rho \equiv 1$ on $M$ from the outset (note that the function $v$ in \cite{CLL23} corresponds to $-\ln u$ in our notation). Ultimately, one arrives at equation (62) therein, but with all appearances of $\rho$ replaced by 1 and all terms involving derivatives of $\rho$ now absent. The crucial point is that the positive term involving $|\nabla v|^4$ in (62) is unaffected, and thus one obtains the desired estimate as in \cite{CLL23}.\medskip 
	
	Now suppose $x_0\in\partial M$. Since $u$ is constant on $\partial M$, all tangential derivatives of $u$ vanish on $\partial M$, and thus we only need to bound $\nabla_{\nu} u$ on $\partial M$, where $\nu$ is the inward pointing unit normal to $\partial M$. As explained in the proof of Proposition \ref{121}, $u \geq v$ where $v$ is the solution to \eqref{150}, and hence for $x\in M$ we have
	\begin{align*}
	\frac{u(x) - u(x_0)}{d_{g_0}(x,x_0)} = \frac{u(x) - v(x_0)}{d_{g_0}(x,x_0)} \geq \frac{v(x) - v(x_0)}{d_{g_0}(x,x_0)},
	\end{align*}
	which implies $\nabla_\nu u(x_0) \geq \nabla_{\nu} v(x_0)$.\medskip 
	
	It remains to prove an upper bound for $\nabla_{\nu} u$, for which we appeal to Proposition \ref{4'}; the method is similar to our proof of \cite[Proposition 3.3]{DN23}. First we attach a collar neighbourhood $N$ to $\partial M$ such that $g_0$ extends smoothly to a metric on $M\cup N$, which we also denote by $g_0$. We then let 
	\begin{equation*}
	D = \inf_{x\in \partial M} \operatorname{d}_{g_0}(x,\partial(M\cup N))
	\end{equation*}
	denote the thickness of $N$. Fix $x_0\in\partial M$, let $\delta$ be as at the start of the proof, let $m = \sup_M u$ and take an annulus $A_{r_-, r_+}(y)$ such that:\medskip 
	\begin{enumerate}
		\item $y\in N$,
		\item $r_- + r_+< D$ (i.e.~the annulus is contained in $M\cup N$)
		\item $\mathbb{S}_{r_-}(y)\cap \partial M = \{x_0\}$ (i.e.~the inner boundary of the annulus touches $\partial M$ at $x_0$),
		\item The closed ball $\overline{B_{r_+}(y)}$ is contained in a single normal coordinate chart $(U, \zeta)$ mapping $y$ to the origin,
		\item $\frac{r_-}{\sqrt{1+\delta}} = \frac{r_+}{\sqrt{1+m}} \eqdef R$, where $R \leq \overline{R}$ is sufficiently small so that Proposition \ref{4'} applies (here we are implicitly identifying the annulus with its image under $\zeta$, which is possible by Property 4).
	\end{enumerate}
	In what follows, we continue to implicitly make the identification between $A_{r_-, r_+}(y)$ and its image under $\zeta$. We wish to apply the comparison principle in $A_{r_-, r_+}(y)\cap M$, whose boundary can be decomposed as $(B_{r_+}(y)\cap \partial M) \cup (\mathbb{S}_{r_+}(y)\cap M)$. The function $v$ on $A_{r_-, r_+}(y)$, as defined in \eqref{3'}, satisfies $u \leq v$ on $B_{r_+}(y)\cap \partial M$, since $v(x_0) = \delta = u(x_0)$ and $v$ is radially increasing. Since $m=\sup_M u$, we also have $u \leq v$ on $\mathbb{S}_{r_+}(y) \cap M$. The comparison principle then implies that $u \leq v$ in $A_{r_-, r_+}(y)\cap M$. Therefore, for $x\in A_{r_-, r_+}(y)\cap M$ we have
	\begin{align*}
	\frac{u(x) - u(x_0)}{d_{g_0}(x,x_0)} = \frac{u(x) - v(x_0)}{d_{g_0}(x,x_0)} \leq \frac{v(x) - v(x_0)}{d_{g_0}(x,x_0)},
	\end{align*}
	which implies $\nabla_{\nu} u(x_0) \leq \nabla_{\nu} v(x_0)$, as required. 
\end{proof}

\subsection{Proof of Theorem \ref{F}}

We now give the proof of Theorem \ref{C''} (equivalently Theorem \ref{F}) using the continuity method: 

\begin{proof}[Proof of Theorem \ref{C''}]
	We first consider the case $\tau_0<1$. Fix $\alpha\in(0,1)$ and define 
	\begin{align*}
	S = \{\tau\in[0,\tau_0]: \eqref{105''} \text{ admits a solution in }C^{2,\alpha}(M)\}.
	\end{align*}
	Since \eqref{105''} admits a unique smooth solution when $\tau=0$ by Lemma \ref{141}, $S$ is non-empty. Moreover, it is well-known (see e.g.~\cite[Corollary 2.5]{GV03b}) that the linearised operator is invertible as a map from $C^{2,\alpha}(M)$ to $C^{0,\alpha}(M)$, and thus $S$ is open. To see that $S$ is closed, it suffices to show that solutions to \eqref{105''} admit a $C^{2,\alpha'}$ estimate for any $\alpha\leq \alpha'<1$ uniformly in $\tau\in[0,\tau_0]$. The $C^0$ estimate is precisely the assertion of Proposition \ref{121}. We may then apply Proposition \ref{123} to obtain the gradient estimate. Since $\tau_0<1$, we may then apply the following global Hessian estimate due to Guan \cite[Theorem 3.2]{Guan08} (see also the earlier work of Gursky \& Viaclovsky \cite{GV03b} in the case $(f,\Gamma) = (\sigma_k^{1/k},\Gamma_k^+)$ for $k\geq 2$): if $\tau<1$ and $u\in C^4(M)$ satisfies \eqref{55} in $M$, then
	\begin{align*}
	|\nabla^2_{g_0} u|_{g_0} \leq C \quad \text{in }M, 
	\end{align*}
	where $C$ is a constant depending on $n, f, \Gamma, (1-\tau)^{-1}$ and upper bounds for $\|g_0\|_{C^4(M)}, \|\psi\|_{C^2(M)}$, $\|\phi\|_{C^2(\partial M)}$ and $\|\ln u\|_{C^1(M)}$. See also the proof of \cite[Proposition 3.9]{DN23} for more details. With the $C^2$ estimate established, \eqref{105''} becomes uniformly elliptic, and since we assume $f$ is concave in \eqref{23'}, we may then apply the regularity theory of Evans-Krylov \cite{Ev82, Kry82, Kry83} to obtain a $C^{2,\alpha'}$ estimate for any $\alpha\leq \alpha'<1$. Therefore $S$ is closed and hence $S = [0,\tau_0]$. Higher regularity follows from classical Schauder theory, and uniqueness is a consequence of the comparison principle (see e.g.~\cite[Proposition 3.7]{DN23}).\medskip 
	
	We now consider the case $\tau_0 = 1$. In this case, the same argument as above yields a unique smooth solution $u_\tau$ to \eqref{105''} for each $\tau<1$, and moreover the sequence is uniformly bounded in $C^1(M)$ by Proposition \ref{121} and Proposition \ref{123}. Therefore, along a sequence $\tau_i\rightarrow 1$, these solutions converge uniformly to some $u\in C^{0,1}(M)$. The fact that $u$ is a viscosity solution to \eqref{105''} with $\tau=1$ follows from exactly the same argument as in the proof of \cite[Theorem 1.3]{LN20b} -- we omit the details here. 
\end{proof}

\subsection{Proof of Theorem \ref{C}}\label{s100}

We now give the proof of Theorem \ref{C'}, from which Theorem \ref{C} follows immediately as explained at the start of the section. 

\begin{proof}[Proof of Theorem \ref{C'}]
	With Theorem \ref{C''} established and in view of the recent gradient estimates of Chu, Li \& Li \cite{CLL23}, the proof of Theorem \ref{C'} now follows from an adaptation of the arguments \cite[Section 4]{DN23}, as we describe now. \medskip 
	
	Let us first consider the case $\tau_0<1$. For each $\delta>0$, let $u_\delta$ denote the smooth solution to \eqref{105''} with $\tau=\tau_0$, whose existence is guaranteed by Theorem \ref{C''}. We wish to take $\delta\rightarrow 0$ to obtain a smooth solution to \eqref{105'} with $\tau=\tau_0$. By the comparison principle, $u_{\delta_2}\leq u_{\delta_1}$ if $\delta_2<\delta_1$. By Proposition \ref{416}, $\| \ln u_\delta \|_{C^0(K)}$ is bounded independently of $\delta$ for any compact $K\Subset M\backslash \partial M$. With this interior two-sided $C^0$ estimate established, we may then appeal to Chu, Li \& Li \cite[Theorem 7.1]{CLL23} (alternatively one may appeal to the earlier gradient estimate of Guan \cite[Theorem 2.1]{Guan08}, which applies since we are currently assuming $\tau_0<1$) to obtain the gradient estimate on $K$. The Hessian estimate on $K$ then follows from Guan \cite[Theorem 3.1]{Guan08}, which applies since $\tau_0<1$. Finally, higher order estimates  follow from the theory of Evans-Krylov \cite{Ev82, Kry82} and classical Schauder theory. Therefore, one may send $\delta\rightarrow 0$ along a subsequence to obtain a smooth positive solution to \eqref{105'} (this is the direct counterpart to our existence result \cite[Proposition 4.1]{DN23}). The existence of a solution satisfying the desired asymptotics \eqref{117} then follows exactly the same proof as \cite[Proposition 4.4]{DN23}, and the uniqueness of the solution to \eqref{105} then follows exactly the same proof as \cite[Proposition 4.6]{DN23}. This completes the proof of Theorem \ref{C'} in the case $\tau_0<1$.\medskip 
	
	 We now consider the case $\tau_0 = 1$. By our argument above, we first obtain for each $\tau<1$ a smooth solution $u^\tau$ to \eqref{105'}. Since the interior $C^0$ estimate in Proposition \ref{416} and the interior gradient estimate of Chu, Li \& Li \cite{CLL23} depending on a two-sided $C^0$ bound are uniform with respect to $\tau \leq 1$, it follows that a subsequence $\{u^{\tau_i}\}$ converges locally uniformly in $C^{0,\alpha}$ to some $u\in C_{\operatorname{loc}}^{0,1}(M,g_0)$ for each $\alpha \in (0,1)$. The fact that $u$ is a viscosity solution to \eqref{105'} follows from exactly the same argument as in the proof of \cite[Theorem 1.4]{LN20b}. Finally, the fact that $u$ satisfies the desired asymptotics \eqref{117} and is maximal follows exactly the same argument as given in \cite[Section 4.4]{DN23}. This completes the proof of Theorem \ref{C'}. 
\end{proof}

\begin{appendices}

	\section{Appendix: Proof of Theorem \ref{19}}\label{appa}

	In this appendix we give the proof of Theorem \ref{19}, which follows the line of arguments in \cite{AKKLT04}; we also provide some additional details at some points in the proof (in particular when controlling the size of the image of the harmonic coordinate charts).\medskip 
	
	We will need the following notion of harmonic radius:
	
	\begin{defn}\label{58}
		Let $(M, g)$ be an $n$-dimensional ($n\geq 3$) Riemannian manifold with non-empty boundary $\partial M$. For $\ell\in\mathbb{N}\cup\{0\}$, $\sigma\in(0,1)$, $Q>1$ and $p\in M$, we call $r^{\ell+\sigma}_{\mathrm{har}}(p,g,Q)\in(0,\infty]$ the \textit{$C^{\ell, \sigma}$-harmonic radius of $(M,g)$ at $p$} if it is the largest number such that for all $\rho<r_{\mathrm{har}}^{\ell+\sigma}$, the two statements in Definition \ref{17} hold when restricted to harmonic coordinate charts. If $M$ is compact, we define the \textit{$C^{\ell, \sigma}$-harmonic radius of $(M,g)$} by 
		\begin{align*}
		r_{\mathrm{har}}^{\ell+\sigma}(M,g,Q) \defeq \inf_{p\in M} r_{\mathrm{har}}^{\ell+\sigma}(p,g,Q)>0. 
		\end{align*}
		Equivalently, $r_{\mathrm{har}}^{\ell+\sigma}(M,g,Q)$ is the largest number such that $(M,g)\in\mathcal{N}^n_{\mathrm{har}}(\ell+\sigma,\rho, Q)$ for all $\rho< r_{\mathrm{har}}^{\ell+\sigma}(M,g,Q)$. 
	\end{defn}
	
	It is by now standard to see that the $C^{\ell,\sigma}$-harmonic radius at a given point $p$ is well-defined and positive; in fact, Theorem \ref{19} is essentially a quantitative version of this assertion. \medskip 
	
	We also need to extend Definition \ref{401} to the notion of pointed $C^{\ell,\sigma}$-convergence: 
	
	\begin{defn}
		A sequence of smooth pointed Riemannian manifolds $\{(M_k, g_k, p_k)\}$ converges in the pointed $C^{\ell,\sigma}$-topology to $(M, g, p)$ if $M$ is a smooth manifold, $g$ is a $C^{\ell,\sigma}$ Riemannian metric on $M$, $p\in M$, and the following holds for sufficiently large $k$. There exist:
		\begin{enumerate}
			\item Radii $\rho_k<\sigma_k$ with $\rho_k\rightarrow\infty$ and compact sets $\overline{\Omega}_k \subset M_k,\, \overline{\mathcal{V}}_k\subset M$ such that
			\begin{align*}
			B_{\rho_k}(p_k)\subset\overline{\Omega}_k \subset B_{\sigma_k}(p_k), \quad B_{\rho_k}(p)\subset \overline{\mathcal{V}}_k \subset B_{\sigma_k}(p) \quad \text{for all }k,
			\end{align*}
			\item Diffeomorphisms $F_k:\overline{\mathcal{V}}_k \rightarrow \overline{\Omega}_k$ which restrict to diffeomorphisms from $\overline{\mathcal{V}}_k\cap \partial M$ to $\overline{\Omega}_k \cap \partial M_k$ and satisfy $F_k^{-1}(p_k)\rightarrow p$,
			\item A locally finite collection of bounded charts on $M$, such that on each chart $(F_k^* g_k)_{ij}\rightarrow g_{ij}$ in the $C^{\ell,\sigma}$-topology.
		\end{enumerate}
	\end{defn}
	
	\begin{defn}\label{57}
		Under the same set-up as Definition \ref{17}, define $\mathcal{N}_*^n(\ell+\sigma,\rho,Q)$ in the same way as $\mathcal{N}^n(\ell+\sigma,\rho,Q)$ but in the class of pointed Riemannian manifolds. 
	\end{defn}

One then has the following extension of Theorem \ref{20} (once again we refer to \cite[Theorem 11.3.6]{Pet16} for the proof):

\begin{thm}\label{20''}
	Let $n\geq 3$, $\ell\in\mathbb{N}\cup\{0\}$, $0<\sigma<\sigma'<1$, $\rho\in(0,\infty)$ and $Q>1$. Then $\mathcal{N}_*^n(\ell+\sigma,\rho,Q)$ is precompact in the pointed $C^{\ell,\sigma'}$-topology. 
\end{thm}
	
We now give the proof of Theorem \ref{19}:
	
	\begin{proof}[Proof of Theorem \ref{19}]
		We suppose for a contradiction that there exists a sequence $\{(M_k, \widetilde{g}_k)\}$ $\subset\mathcal{M}^n_\sigma(R_0, i_0, S_0,\infty)$ such that $r_{\mathrm{har}}^{1+\sigma}(M_k, \widetilde{g}_k, Q)\eqdef \ep_k\rightarrow 0$ as $k\rightarrow\infty$. We then define the rescaled metrics $g_k = \ep_k^{-2}\widetilde{g}_k$, which satisfy
		\begin{align}\label{402}
		r_{\mathrm{har}}^{1+\sigma}(M_k, g_k, Q) = 1 \quad \text{for all }k, 
		\end{align}
		\begin{align}\label{27}
		\|\operatorname{Ric}_{(M_k, g_k)}\|_{L^\infty(M_k)}, \, \|\operatorname{Ric}_{(\partial M_k, g_k)}\|_{L^\infty(\partial M_k)}, \|H_{(\partial M_k, g_k)}\|_{C^{0,\sigma}(\partial M_k, g_k)} \longrightarrow 0 \quad \text{as }k\rightarrow\infty
		\end{align}
		and
		\begin{align}\label{28}
		i_{(M_k, g_k)}, \, i_{(\partial M_k, g_k)}, \, i_{b,(\partial M_k, g_k)} \longrightarrow \infty \quad \text{as }k\rightarrow \infty. 
		\end{align}
		By \eqref{402}, for each $k$ there exists $p_k \in M_k$ such that
		\begin{align}\label{23}
		r_{\mathrm{har}}^{1+\sigma}(p_k, g_k, Q) < \frac{3}{2}. 
		\end{align}
		
		Now, by Theorem \ref{20''} there exists a subsequence (which we still index by $k$) and a pointed Riemannian manifold $(M, g, p)$ such that $(M_k, g_k, p_k) \rightarrow (M, g, p)$ in the pointed $C^{1,\sigma'}$-topology for each $\sigma'<\sigma$. Following \cite{AKKLT04, And90}, we then have two possibilities:
		\begin{enumerate}
			\item $\operatorname{dist}_{g_k}(p_k,\partial M_k)\rightarrow\infty$ as $k\rightarrow\infty$, in which case $(M,g)$ is isometric to $\mathbb{R}^n$ with the Euclidean metric, or
			\item $\operatorname{dist}_{g_k}(p_k,\partial M_k)\leq K < \infty$, in which case $(M,g)$ is isometric to $\overline{\mathbb{R}_+^n}$ with the Euclidean metric. 
		\end{enumerate}
	We refer to \cite[Lemma 3.2.2]{AKKLT04} for the proof of this dichotomy, which builds on an argument on pages 434-435 of Anderson \cite{And90}. We point out that the proof of this dichotomy in \cite{AKKLT04} only requires $\|H_{(\partial M_k,g_k)}\|_{L^\infty(\partial M_k)}\rightarrow 0$ as $k\rightarrow \infty$, and therefore applies in our setting. The rest of the proof of Theorem \ref{19} proceeds according to whether we are in Case 1 or 2, so we now consider these cases separately. \bigskip 
		
		\noindent\underline{Case 1.} Since $\operatorname{dist}_{g_k}(p_k,\partial M_k)\rightarrow\infty$ as $k\rightarrow\infty$, there exist neighbourhoods $U_k$ of $p_k$ in $M_k$ identified with $B_5\subset\mathbb{R}^n$ such that, under this identification, $p_k=0$ and $g_k\rightarrow\delta$ in $C^{1,\sigma'}(B_5)$. In what follows, the $C^{k,\alpha}$ norms of scalar functions are implicitly measured using the Euclidean distance. Let $\{x^\nu\}_{1\leq \nu \leq n}$ denote the Cartesian coordinates on $\mathbb{R}^n$, and for each $k$ and $1 \leq \nu \leq n$ define $u_k^\nu$ via
		\begin{align}\label{47}
		\begin{cases}
		\Delta_{g_k} u_k^\nu = 0 & \text{in }B_5 \\
		u_k^\nu = x^\nu & \text{on }\partial B_5. 
		\end{cases}
		\end{align}
		It follows from \eqref{47} that
		\begin{align*}
		\begin{cases}
		\Delta_{g_k}(u_k^\nu - x^\nu) = f_k^\nu & \text{in }B_5 \\
		u_k^\nu - x^\nu = 0 & \text{on }\partial B_5, 
		\end{cases}
		\end{align*}
		where
		\begin{align}\label{48}
		f_k^\nu \defeq -\Delta_{g_k} x^\nu = -\frac{1}{\sqrt{\operatorname{det}g_k}}\partial_i \Big(\sqrt{\operatorname{det}g_k}\,g_k^{ij}\partial_j x^\nu\Big) =  -\frac{1}{\sqrt{\operatorname{det}g_k}}\partial_i \Big(\sqrt{\operatorname{det}g_k}\,g_k^{i\nu}\Big).
		\end{align}
		Now, since $g_k\rightarrow \delta$ in $C^{1,\sigma'}(B_5)$, we see that $\|f_k^\nu\|_{C^{0,\sigma'}(B_5)}\rightarrow 0$ as $k\rightarrow \infty$. Moreover, since the leading order term in the Laplace-Beltrami operator $\Delta_{g_k}$ is $g_k^{ij}\partial_i\partial_j$, classical Schauder estimates for non-divergence form elliptic equations (see e.g.~\cite[Theorem 6.6]{GT}) imply
		\begin{align}\label{21}
		\|u_k^\nu - x^\nu\|_{C^{2,\sigma'}(B_5)} \rightarrow 0  \quad \text{as }k\rightarrow\infty. 
		\end{align}
		It follows from \eqref{21} that for sufficiently large $k$, $u_k$ is a local diffeomorphism from $B_5$ onto its image. Moreover, 
		\begin{align}\label{211}
		|u_k(y) - u_k(z)| & \geq |y-z| - |(u_k-x)(y) - (u_k-x)(z)| \nonumber \\
		&  \geq |y-z| - \|u_k - x\|_{C^{0,1}(B_5)}|y-z| \nonumber \\
		& \geq \frac{1}{2}|y-z|
		\end{align}
		for sufficiently large $k$ by \eqref{21}, so $u_k$ is injective and hence a diffeomorphism onto its image for large $k$. In particular, $\{u_k^\nu\}_{1\leq \nu \leq n}$ define harmonic coordinates on $B_5$. To assert a lower bound on the harmonic radius later in the proof, we also need to control the size of the image of $u_k$. We claim that $u_k(B_5) = B_5$ (and hence, by the boundary data in \eqref{47}, $u_k(\overline{B}_5) = \overline{B}_5$) for large $k$. To see this, first observe by \eqref{47} that
		\begin{align}\label{212}
		\begin{cases}
		\Delta_{g_k}|u_k|^2 = 2\sum_\nu (u_k^\nu \Delta_{g_k}u_k^\nu + |du_k^\nu|^2) \geq 0 & \text{in }B_5 \\
		|u_k|^2 = |x|^2 = 25 & \text{on }\partial B_5.
		\end{cases}
		\end{align}
		The strong maximum principle therefore implies $u_k(B_5) \subseteq B_5$. On the other hand, by the boundary data in \eqref{47}, $u_k(\partial B_5) = \partial B_5$. Therefore $u_k(B_5)$ is both open and closed in $B_5$, and since it is non-empty, it is therefore equal to $B_5$ as claimed. In particular, in combination with \eqref{21} we also see that 
		\begin{align}\label{411}
		u_k^{-1} \text{ converges to the identity map on }\overline{B}_5 \text{ as }k\rightarrow\infty.
		\end{align}

		 Denoting by $g_{(k)}^{ij} = g_k^{-1}(du_k^i, du_k^j)$ the components of the inverse metric $g_k^{-1}$ with respect to the harmonic coordinates, it follows from \eqref{21} that 
		\begin{align*}
		\|g^{ij}_{(k)} - \delta^{ij}\|_{C^{1,\sigma'}(B_5)}\rightarrow 0 \quad \text{as }k\rightarrow\infty,
		\end{align*}
		and hence the components $g_{ij}^{(k)}$ of the metric $g_k$ with respect to the harmonic coordinates satisfy
		\begin{align}\label{22}
		\|g_{ij}^{(k)} - \delta_{ij}\|_{C^{1,\sigma'}(B_5)}\rightarrow 0 \quad \text{as }k\rightarrow\infty.
		\end{align}
		We now wish to upgrade the convergence in \eqref{22} to $C^{1,\sigma}$ convergence on a smaller ball, namely we claim that
		\begin{align}\label{25}
		\|g_{ij}^{(k)} - \delta_{ij}\|_{C^{1,\sigma}(B_3)}\rightarrow 0 \quad \text{as }k\rightarrow\infty.
		\end{align}
		Once \eqref{25} is obtained, it follows from \eqref{411} that $r_{\mathrm{har}}^{1+\sigma}(p_k,g_k,Q) \geq 2$ for sufficiently large $k$, which contradicts \eqref{23} and thus completes Case 1 of the proof.\medskip 
		
		To prove \eqref{25}, we use the well-known fact (see e.g.~\cite{DK81}) that the components $g_{ij}^{(k)}$ of the metric tensor in harmonic coordinates $\{u_k^\nu\}_{1\leq \nu \leq n}$ satisfy the elliptic equation 
		\begin{align*}
		\Delta_{g_k}(g_{ij}^{(k)} - \delta_{ij}) = F_{ij}^{(k)} \defeq B_{ij}(g^{(k)}_{lm}, \nabla g_{lm}^{(k)}) - 2\operatorname{Ric}_{ij}^{(k)},
		\end{align*}
		where $B_{ij}$ is smooth in both arguments. Now, by \eqref{22} and the $L^\infty$ decay of the Ricci curvature in \eqref{27}, we know that $\|F_{ij}^{(k)}\|_{L^\infty(B_5)}\rightarrow 0$. It follows from interior $W^{2,p}$ elliptic regularity (see e.g.~\cite[Theorem 9.11]{GT}) that $\|g_{ij}^{(k)}-\delta_{ij}\|_{W^{2,p}(B_4)}\rightarrow 0$ for any $p<\infty$, and hence by the Morrey embedding theorem $\|g_{ij}^{(k)}-\delta_{ij}\|_{C^{1,s}(B_3)}\rightarrow 0$ for each $s<1$. In particular, \eqref{25} holds. \bigskip
		
		\noindent\underline{Case 2.} Let $q_k\in\partial M_k$ be the point in $\partial M_k$ closest to $p_k$, which is uniquely defined for sufficiently large $k$ due to \eqref{28}. By the pointed convergence of $(M_k,g_k,p_k)$ to $(M,g,p)$ (where we recall that $(M,g)$ isometric to $\overline{\mathbb{R}_+^n}$ with the Euclidean metric), after possibly passing to a further subsequence there exists $q\in \partial M = \partial\mathbb{R}^n_+$ such that $\operatorname{dist}_{g_k}(p_k,q_k)\rightarrow\operatorname{dist}_g(p,q)$. Fixing $L \geq \operatorname{dist}_g(p,q) + 2$, we therefore have neighbourhoods $U_k$ of $q_k$ in $M_k$ identified with $B_{L+5}^+ =  B_{L+5}\cap\{x_n\geq 0\}$ such that, under this identification, $q_k=0$ and $g_k\rightarrow\delta$ in $C^{1,\sigma'}(B_{L+5}^+)$. We let $\{x^\nu\}_{1 \leq \nu \leq n}$ denote the Cartesian coordinates on $\overline{\mathbb{R}_+^n}$, so that in particular $\partial\mathbb{R}_+^n = \{x_n=0\}$. \medskip

		We construct harmonic coordinates in $B_{L+3}^+$ as follows. For $r>0$, we decompose the boundary of $B^+_r$ as 
		\begin{align*}
		\partial B^+_r = \partial' B_r^+ \cup \partial'' B_r^+,
		\end{align*}
		where $\partial' B_r^+ = B_r^+ \cap \{x_n=0\}$ and $\partial '' B_r^+ = \partial B^+_r  \backslash \partial' B_r^+$ is a hemisphere. For each $k$ and $1 \leq \nu \leq n-1$, we first define harmonic coordinates $v_k^\nu$ on $\partial' B_{L+5}^+$ as in Case 1 via 
		\begin{align}\label{413}
		\begin{cases}
		\Delta_{\partial M_k} v_k^\nu = 0 & \text{in }\partial' B_{L+5}^+ \\
		v_k^\nu = x^\nu & \text{on }\partial (\partial' B_{L+5}^+).
		\end{cases}
		\end{align}
		Analogously to \eqref{21}, we have
		\begin{align}\label{26}
		\|v_k^\nu - x^\nu\|_{C^{2,\sigma'}(\partial' B_{L+5}^+)}\rightarrow 0 \quad \text{as }k\rightarrow \infty \text{ for }1 \leq \nu \leq n-1
		\end{align}
		and hence by the same reasoning as in Case 1, $\{v_k^\nu\}_{1\leq \nu \leq n-1}$ do indeed define harmonic coordinates on $\partial' B_{L+5}^+$. For each $k$ and $1 \leq \nu \leq n-1$, we now extend the $v_k^{\nu}$ to functions $u_k^\nu$ on $B_{L+5}^+$ via the boundary value problem:
		\begin{align}\label{412}
		\begin{cases}
		\Delta_{g_k} u_k^\nu = 0 & \text{in } \operatorname{int}(B_{L+5}^+) \\
		u_k^\nu = v_k^\nu & \text{on }\partial' B_{L+5}^+ \\
		u_k^\nu = x^\nu & \text{on }\partial'' B_{L+5}^+. 
		\end{cases}
		\end{align}
		It follows that
		\begin{align*}
		\begin{cases}
		\Delta_{g_k}(u_k^\nu - x^\nu) = f_k^\nu & \text{in }\operatorname{int}(B_{L+5}^+) \\
		u_k^\nu - x^\nu = \phi_k^\nu & \text{on }\partial B_{L+5}^+,
		\end{cases}
		\end{align*}
		where $f_k^\nu$ is defined in \eqref{48} and satisfies $\|f_k^\nu\|_{C^{0,\sigma'}}(B_{L+5}^+)\rightarrow 0$ as $k\rightarrow \infty$, and
		\begin{align*}
		\phi_k^\nu \defeq \begin{cases}
		v_k^\nu - x^\nu & \text{on }\partial' B_{L+5}^+ \\
		0 & \text{on }\partial'' B_{L+5}^+  . 
		\end{cases}
		\end{align*} 
		Using \eqref{26}, it follows that
		\begin{align*}
		\|\phi_k^\nu\|_{C^{2,\sigma'}(\partial' B_{L+5}^+)}\rightarrow 0 \quad \text{and} \quad \|\phi_k^\nu\|_{L^\infty(\partial B_{L+5}^+)}\rightarrow 0 \quad \text{as }k\rightarrow\infty \text{ for }1 \leq \nu \leq n-1. 
		\end{align*}
		Elliptic regularity on the Lipschitz domain $B_{L+5}^+$ therefore implies 
		\begin{align*}
		\|u_k^\nu - x^\nu\|_{L^\infty(B_{L+5}^+)} \rightarrow 0 \quad \text{as }k\rightarrow\infty \text{ for }1 \leq \nu \leq n-1
		\end{align*}
		and 
		\begin{align}\label{29}
		\|u_k^\nu - x^\nu\|_{C^{2,\sigma'}(B_{L+4}^+)} \rightarrow 0 \quad \text{as }k\rightarrow\infty \text{ for }1 \leq \nu \leq n-1. 
		\end{align}

		To construct the final harmonic coordinate $u_k^n$, we simply solve 
		\begin{align*}
		\begin{cases}
		\Delta_{g_k} u_k^n = 0 & \text{in } \operatorname{int}(B_{L+5}^+) \\
		u_k^n = x^n & \text{on }\partial B_{L+5}^+,
		\end{cases}
		\end{align*}
		from which it follows using arguments as above that \eqref{29} also holds for $\nu=n$. Combining the above, we therefore obtain
		\begin{align}\label{30}
		\|u_k^\nu - x^\nu\|_{C^{2,\sigma'}(B_{L+4}^+)} \rightarrow 0 \quad \text{as }k\rightarrow\infty \text{ for }1 \leq \nu \leq n,
		\end{align}
		and hence by the same reasoning as in Case 1 (see \eqref{21}, \eqref{211}), $\{u_k^\nu\}_{1 \leq \nu\leq n}$ do indeed define harmonic coordinates on $B_{L+4}^+$, with $\{u_k^n=0\}$ corresponding to $\{x^n=0\}$ for large $k$. In particular, for such $k$, $u_k$ is a diffeomorphism from $B_{L+4}^+$ into its image. Once again, we must control the size of the image of $u_k$ in order to gain control on the harmonic radius later in the proof, and indeed we claim that\footnote{It follows immediately from \eqref{30} that $u_k^{-1}(B_{L+3}^+)\subset B_{L+4}^+$ for large $k$. The claim in \eqref{404} is that the whole of $B_{L+3}^+$ belongs to the image.}
		\begin{align}\label{404}
		B_{L+3}^+ \subset u_k(B_{L+4}^+)\quad \text{for sufficiently large }k.
		\end{align}
		Recall $u_k|_{\partial' B_{L+5}^+} = v_k$ by \eqref{412} and $v_k$ is defined via \eqref{413}, which mirrors exactly the set-up \eqref{47} in Case 1. Arguing as in Case 1 (see \eqref{211} and \eqref{212}), we see that $v_k(\partial' B_{L+5}^+)= \partial' B_{L+5}^+$ and $v_k^{-1}$ converges to the identity map on $\partial' B_{L+5}^+$ as $k\rightarrow \infty$. It follows that for sufficiently large $k$, $\partial' B_{L+3}^+ \subset u_k(\partial' B_{L+4}^+)$. To complete the proof of \eqref{404}, it therefore remains to show $\operatorname{int}(B_{L+3}^+)\subset u_k(\operatorname{int}(B_{L+4}^+))$, or equivalently 
		\begin{align}\label{406}
		\operatorname{int}(B_{L+3}^+) \cap u_k(\operatorname{int}(B_{L+4}^+)) = 	\operatorname{int}(B_{L+3}^+) \quad \text{for sufficiently large }k.
		\end{align}
		By \eqref{30} and the fact $u_k(\operatorname{int}(B_{L+4}^+))\subset\{x_n>0\}$, the intersection on the LHS of \eqref{406} is clearly non-empty for sufficiently large $k$, and so to prove \eqref{406} it suffices to show that this intersection is both open and closed in $\operatorname{int}(B_{L+3}^+)$ for large $k$. In what follows, we fix $k$ sufficiently large so that $u_k$ is a diffeomorphism from $B_{L+4}^+$ onto its image and $u_k(\partial'' B_{L+4}^+) \subset \mathbb{R}^n\backslash \overline{B_{L+3}^+}$. Openness is clear from the fact that $u_k$ is a diffeomorphism. To see that it is closed in $\operatorname{int}(B_{L+3}^+)$, take a sequence $\{y_i\}\subset \operatorname{int}(B_{L+3}^+) \cap u_k(\operatorname{int}(B_{L+4}^+))$ converging to $y\in\operatorname{int}(B_{L+3}^+)$. We wish to show $y\in \operatorname{int}(B_{L+3}^+) \cap u_k(\operatorname{int}(B_{L+4}^+))$, for which we must establish the existence of $z\in \operatorname{int}(B_{L+4}^+)$ such that $y = u_k(z)$. Let $z_i \in \operatorname{int}(B_{L+4}^+)$ be such that $y_i = u_k(z_i)$ for each $i$. Then, after possibly passing to a subsequence, we have $z_i\rightarrow z\in \overline{\operatorname{int}(B_{L+4}^+)} = \overline{B_{L+4}^+}$ and $y = u_k(z)$. Since $u_k(z) = y\in \operatorname{int}(B_{L+3}^+)$ and since we have chosen $k$ large enough so that $u_k(\partial'' B_{L+4}^+) \subset \mathbb{R}^n\backslash \overline{B_{L+3}^+}$, we must have $z\not\in \partial'' B_{L+4}^+$. Likewise, since $u_k(\partial' B_{L+4}^+)\subset \{x_n=0\}$ and $u_k(z) = y \not\in \{x_n=0\}$, we must also have $z\not\in \partial' B_{L+4}^+$. Therefore $z\in \operatorname{int}(B_{L+4}^+)$, as required.\medskip

		By combining \eqref{404} with \eqref{30}, we therefore have the following assertion:
		\begin{align}\label{407}
		u_k^{-1}|_{B_{L+3}^+} \text{ converges to the natural injection of }B_{L+3}^+ \text{ in }B_{L+4}^+ \text{ as }k\rightarrow\infty. 
		\end{align}

		Denoting by $g^{ij}_{(k)} = g_k^{-1}(d u_k^i, du_k^j)$ the components of the inverse metric $g_k^{-1}$ with respect to the harmonic coordinates, it follows from \eqref{30} that
		\begin{align}\label{414}
		\|g^{ij}_{(k)} - \delta^{ij}\|_{C^{1,\sigma'}(B_{L+4}^+)}\rightarrow 0 \quad \text{as }k\rightarrow\infty
		\end{align}
		and hence the components $g_{ij}^{(k)}$ of the metric $g_k$ with respect to the harmonic coordinates satisfy 
		\begin{align}\label{31}
		\|g_{ij}^{(k)} - \delta_{ij}\|_{C^{1,\sigma'}(B_{L+4}^+)}\rightarrow 0 \quad \text{as }k\rightarrow\infty
		\end{align}
		As in Case 1, we wish to upgrade the convergence in \eqref{31} to $C^{1,\sigma}$ convergence on a smaller half-ball, namely we claim that 
		\begin{align}\label{32}
		\|g_{ij}^{(k)} - \delta_{ij}\|_{C^{1,\sigma}(B_{L+1}^+)}\rightarrow 0 \quad \text{as }k\rightarrow\infty. 
		\end{align}
		Once \eqref{32} is obtained, it follows from \eqref{407} and the fact $L \geq \operatorname{dist}_g(p,q) + 2$ that $r_{\mathrm{har}}^{1+\sigma}(p_k,g_k,Q) \geq 2$ for sufficiently large $k$, which contradicts \eqref{23} and thus completes Case 2 of the proof. \medskip
		
		We split the proof of \eqref{32} into two steps, following a slightly different presentation to that given in \cite{AKKLT04}.\medskip 
		
	\noindent\textbf{Step 1:} 
	\begin{align}\label{39}
	\|g_{(k)}^{ln} - \delta^{ln}\|_{C^{1,\sigma}(B_{L+2}^+)} \rightarrow 0  \quad \text{as }k\rightarrow \infty \text{ for }1 \leq l \leq n. 
	\end{align}
	
	\noindent\textbf{Step 2:}
	\begin{align}\label{41'}
	\|g^{ij}_{(k)} - \delta^{ij}\|_{C^{1,\sigma}(B_{L+1}^+)} \rightarrow 0 \quad \text{as }k\rightarrow\infty \text{ for }1 \leq i,j \leq n-1.
	\end{align}
	
	Once we have carried out Steps 1 and 2, we will have shown 
	\begin{align*}
	\|g^{ij}_{(k)} - \delta^{ij}\|_{C^{1,\sigma}(B_{L+1}^+)}\rightarrow 0 \quad \text{as }k\rightarrow\infty \text{ for }1 \leq i, j \leq n,
	\end{align*}
	from which \eqref{32} follows easily.\medskip

	\noindent\textit{Proof of Step 1.} We start by recalling the well-known fact (see e.g.~\cite{DK81}) that the components $g_{(k)}^{lm}$ of the inverse metric tensor in harmonic coordinates $\{u_k^\nu\}_{1 \leq \nu \leq n}$ satisfy the elliptic equation 
		\begin{align}\label{408}
		\Delta_{g_k}(g_{(k)}^{lm} - \delta^{lm}) = F_{(k)}^{lm} \defeq B^{lm}(g_{(k)}^{ab},\nabla g_{(k)}^{ab}) + 2\operatorname{Ric}_{(k)}^{lm} \quad \text{in }B_{L+4}^+,
		\end{align}
		where $B^{lm}$ is smooth in both entries. Now, by \eqref{414} and the $L^\infty$ decay of the Ricci curvature in \eqref{27}, we have
		\begin{align}\label{35}
		\|F_{(k)}^{ln}\|_{L^\infty(B_{L+4}^+)}\rightarrow 0 \quad \text{as }k\rightarrow\infty. 
		\end{align} 
		On the other hand, as derived in the proof of \cite[Lemma 2.1.2]{AKKLT04}, $g_{(k)}^{ln} - \delta^{ln}$ satisfies the following Neumann boundary conditions on $\partial' B_{L+4}^+$:
		\begin{align}\label{36}
		\nabla_N(g_{(k)}^{ln} - \delta^{ln}) = \begin{cases}
		-2(n-1) H_k g_{(k)}^{nn} & \text{if } l=n \\
		-(n-1)H_k g_{(k)}^{ln} + \frac{1}{2\sqrt{g_{(k)}^{nn}}}g_{(k)}^{li}\partial_i g_{(k)}^{nn} & \text{if }1 \leq l \leq n-1,
		\end{cases}
		\end{align}
		where $N = \frac{\nabla u_k^n}{|\nabla u_k^n|}$ is the upward pointing unit normal field on $\partial' B_{L+4}^+$. By \eqref{402} (which implies the metrics $g_{(k)}$ and their inverses are uniformly bounded in $C^{1,\sigma}$) and the $C^{0,\sigma}$ decay of the mean curvature in \eqref{27}, we have that $\|H_k g_{(k)}^{nn}\|_{C^{0,\sigma}(\partial' B_{L+4}^+)}\rightarrow 0$ as $k\rightarrow\infty$. In combination with \eqref{35}, elliptic regularity (see \cite[Theorem 5.4.1]{AKKLT04}) therefore yields 
		\begin{align}\label{37}
		\|g^{nn}_{(k)} - \delta^{nn}\|_{C^{1,\sigma}(B_{L+3}^+)}\rightarrow 0 \quad \text{as }k\rightarrow\infty. 
		\end{align}
		It then follows that the right hand side in the bottom line of \eqref{36} also converges to 0 in $C^{0,\sigma}(\partial' B_{L+3}^+)$ as $k\rightarrow\infty$, and hence elliptic regularity \cite[Theorem 5.4.1]{AKKLT04} again implies 
		\begin{align}\label{38}
		\|g^{ln}_{(k)} - \delta^{ln}\|_{C^{1,\sigma}(B_{L+2}^+)}\rightarrow 0 \quad \text{as }k\rightarrow\infty \text{ for }1 \leq l \leq n-1.
		\end{align}
		Combining \eqref{37} and \eqref{38}, we arrive at \eqref{39}, which completes Step 1.\medskip 
		
		\noindent\textit{Proof of Step 2.} We only need to show
		\begin{align}\label{410}
		\|g_{(k)}^{ij}-\delta^{ij}\|_{C^{1,\sigma}(\partial' B_{L+2}^+)}\rightarrow 0 \quad \text{as }k\rightarrow\infty \text{ for }1 \leq i,j \leq n-1. 
		\end{align} 
		Indeed, in light of \eqref{31} and the equation \eqref{408} satisfied by the inverse metric components $g^{ij}_{(k)}$ in $B_{L+2}^+$, once \eqref{410} is established the desired estimate \eqref{41'} again follows from elliptic regularity. \medskip 
		
		To prove \eqref{410}, we start by writing
		\begin{align*}
		\nabla u_k^i = \nabla_N  u_k^i + \nabla_T u_k^i, 
		\end{align*}
		where $N = \frac{\nabla u^n_k}{|\nabla u^n_k|} = \frac{\nabla u_k^n}{\sqrt{g^{nn}_{(k)}}}$ is the unit upward pointing normal, $\nabla_N u_k^i = g_k(N, \nabla u_k^i)N$ and $\nabla_T$ denotes the tangential gradient. Note that 
		\begin{align*}
		\nabla_N u_k^i = g_k(N,\nabla u_k^i)N = \frac{1}{\sqrt{g_{(k)}^{nn}}} g_k(\nabla u^n_k, \nabla u^i_k)N = \frac{g_{(k)}^{ni}}{\sqrt{g_{(k)}^{nn}}}N.
		\end{align*}
		Therefore
		\begin{align}\label{409}
		\|g_{(k)}^{ij}-&\delta^{ij}\|_{C^{1,\sigma}(\partial' B_{L+2}^+)}  = \|g_k(\nabla u_k^i,\nabla u_k^j) - \delta^{ij}\|_{C^{1,\sigma}(\partial' B_{L+2}^+)} \nonumber \\
		& \qquad \qquad  \leq \|g_k(\nabla_T u_k^i,\nabla_T u_k^j) - \delta^{ij}\|_{C^{1,\sigma}(\partial' B_{L+2}^+)} + \|g_k(\nabla_N u_k^i, \nabla_N u_k^j)\|_{C^{1,\sigma}(\partial' B_{L+2}^+)} \nonumber \\
		& \qquad \qquad = \|g_k(\nabla_T v_k^i,\nabla_T v_k^j) - \delta^{ij}\|_{C^{1,\sigma}(\partial' B_{L+2}^+)} + \bigg\|\frac{g_{(k)}^{ni}g_{(k)}^{nj}}{g_{(k)}^{nn}}\bigg\|_{C^{1,\sigma}(\partial' B_{L+2}^+)}.
		\end{align}
		Now, the $g_k(\nabla_T v_k^i,\nabla_T v_k^j)$ are the components of the induced metric on $\partial'B_{L+2}^+$ with respect to the harmonic coordinates, and hence the argument in Case 1 tells us that the first term on the last line in \eqref{409} tends to zero as $k\rightarrow\infty$. Finally, the second term on the last line of \eqref{409} tends to zero by Step 1. This proves \eqref{410}, which completes the proof of Step 2 and hence the proof of Theorem \ref{19}. 
	\end{proof}
\end{appendices}

\footnotesize
\bibliography{references}{}
\bibliographystyle{siam}
\end{document}